\documentclass[a4paper, 12pt]{article}

\usepackage[english]{babel}
\usepackage{amsmath}
\usepackage{amsthm}
\usepackage{amssymb}
\usepackage{mathabx}
\usepackage{hyperref}
\usepackage{cleveref}
\usepackage{graphicx}
\usepackage{tikz}
\usepackage{pgfplots}
\usepackage{todonotes}

\newtheorem{theorem}{Theorem}
\newtheorem{proposition}[theorem]{Proposition}
\newtheorem{lemma}[theorem]{Lemma}
\newtheorem{corollary}[theorem]{Corollary}
\newtheorem{remark}[theorem]{Remark}
\newtheorem{definition}[theorem]{Definition}

\newcommand{\R}{\mathbb{R}}
\newcommand{\T}{\mathbb{T}}
\newcommand{\N}{\mathbb{N}}
\newcommand{\Z}{\mathbb{Z}}
\renewcommand{\d}{\mathrm{d}}
\renewcommand{\div}{\mathrm{div}}
\newcommand{\id}{\mathrm{id}}
\newcommand{\DiffSob}{\mathcal D}
\newcommand{\DiffSobId}{\mathcal D_{\id}}
\newcommand{\DiffSobIdLift}{\widetilde{\mathcal D}_{\id}}
\newcommand{\Diff}{\mathrm{Diff}}
\newcommand{\Hom}{\mathrm{Hom}}
\newcommand{\dist}{\mathrm{dist}_{\DiffSobId^m}}
\newcommand{\Riesz}{\mathcal{R}}
\renewcommand{\L}{\mathcal{L}}
\newcommand{\onesymbol}{\mathrm{I}}
\newcommand{\adjointMap}{\mathrm{Ad}}
\newcommand{\adjointDer}{\mathrm{ad}}

\newcommand{\trunc}[1][R]{\mathcal{T}_{#1}}
\newcommand{\truncFourier}[2]{\widehat{#1}^{#2}}
\newcommand{\truncSpace}{H^m_R}
\newcommand{\Span}{\mathrm{span}}
\newcommand{\Id}{{\mathcal{I}}}
\newcommand{\IPHm}[2]{(\!(#1,#2)\!)_{H^m}} 
\newcommand{\NHm}[1]{\vvvert#1\vvvert_{H^m}} 

\begin{document}

\title{Convergence of spectral discretization for the flow of diffeomorphisms}
\author{Benedikt Wirth}
\date{}

\maketitle

\begin{abstract}
The Large Deformation Diffeomorphic Metric Mapping (LDDMM) or flow of diffeomorphism is a classical framework in the field of shape spaces and is widely applied in mathematical imaging and computational anatomy.
Essentially, it equips a group of diffeomorphisms with a right-invariant Riemannian metric, which allows to compute (Riemannian) distances or interpolations between different deformations.
The associated Euler--Lagrange equation of shortest interpolation paths is one of the standard examples of a partial differential equation that can be approached with Lie group theory
(by interpreting it as a geodesic ordinary differential equation on the Lie group of diffeomorphisms).
The particular group $\DiffSob^m$ of Sobolev diffeomorphisms is by now sufficiently understood to allow the analysis of geodesics and their numerical approximation.
We prove convergence of a widely used Fourier-type space discretization of the geodesic equation.
It is based on a regularity estimate, for which we also provide a new proof:
Geodesics in $\DiffSob^m$ preserve any higher order Sobolev regularity of their initial velocity.
\end{abstract}

%

\section{Introduction and main results}
There are a number of partial differential equations (PDEs) that can be interpreted as a geodesic equation (the Euler--Lagrange equation satisfied by locally shortest paths) on an infinite-dimensional Lie group.
This viewpoint started with the seminal work by Arnol'd in the 1960s on hydrodynamics \cite{Ar66},
who for instance interpreted the Euler equations of inviscid incompressible fluid flow as the geodesic equations on the Lie group of volume-preserving diffeomorphisms endowed with a right-invariant $L^2$-metric.
The great advantage of such a viewpoint is that it often makes these PDEs amenable to an analysis via ordinary differential equation (ODE) techniques
(corresponding sufficient conditions are e.g.\ given in \cite{Ko17}).
It is also exploited for numerics, e.g.\ by devising efficient, structure-preserving solvers based on Hamiltonian system integrators.
Other examples of PDEs that fit into this framework include Burgers' equation, the Camassa-Holm equation, and the KdV equation, see \cite{MiPr10,KhWe09,Ko17} and the references therein.

Yet another classical prototype example is the so-called EPDiff equation
\begin{equation}\label{eqn:EPDiffIntro}
\dot\rho_t=-(\div(\rho_t\otimes v_t)+(Dv_t)^T\rho_t)
\qquad\text{with }
v_t=\Riesz\rho_t=\L^{-1}\rho_t
\end{equation}
for $\L$ a self-adjoint differential operator such as $\L=(1-\Delta)^m$ with $m\geq0$ ($t$ denotes time, the dot time differentiation, and $\Riesz=\L^{-1}$).
The close similarity to the Euler equations becomes apparent for $m=0$, in which case the equation reads $\dot v_t+\div(v_t\otimes v_t)+\nabla\frac{|v_t|^2}2=0$
(the incompressible Euler equations just differ by the additional incompressibility constraint $\div v_t=0$ and the replacement of the internal energy $\frac{|v_t|^2}2$ by the pressure $p$).
This EPDiff equation occurs as the Euler--Lagrange equation when trying to deform (or rather transport) a given image into another one by a time-dependent velocity field $v_t$
with least possible energy $\int_0^1\langle\L v_t,v_t\rangle\,\d t$.
Therefore it is used a lot in computational anatomy, where medical images of a patient have to be mapped to an annotated template image
and where this framework for dealing with deformations is known as Large Deformation Diffeomorphic Metric Mapping (LDDMM).

This EPDiff equation actually turns out to be the geodesic equation on the group $\DiffSob^m$ of diffeomorphisms (e.g.\ of the unit cube) of Sobolev regularity $m$, endowed with a right-invariant Riemannian Sobolev metric:
A geodesic (i.e.\ a locally shortest path) $t\mapsto\phi_t$ of diffeomorphisms can be written as the so-called flow of a velocity field $v_t$, i.e.\ as solution of
\begin{equation*}
\dot\phi_t=v_t\circ\phi_t
\end{equation*}
(in the language of fluid mechanics, $\phi_t$ describes the motion in Lagrangian coordinates, while $v_t$ is the Eulerian description of the motion),
and this velocity field $v_t$ satisfies the EPDiff equation (where $\L$ is related to the employed Riemannian metric $g$ via $g_\id(v,w)=\langle\L v,w\rangle$).

The EPDiff equation represents a typical model setting for applying Lie group techniques to PDEs, see e.g.\ \cite{TrYo15,MiPr10}.
One of the reasons is that by now the group of Sobolev diffeomorphisms and its geometric properties are quite well understood \cite{MiPr10,BrVi17,GuRaRuWi23}
and that the numerical implementation is straightforward.
In particular, \eqref{eqn:EPDiffIntro} lends itself to a space discretization by truncated Fourier series:
For instance, the authors of \cite{ZhFl19} propose a scheme on the $d$-dimensional torus $\T^d$ in which $\rho_t$ and $v_t$ are approximated by bandlimited functions $P_t$ and $V_t$ (numerically represented by their finite Fourier series)
and both right-hand sides in \eqref{eqn:EPDiffIntro} are essentially just truncated in Fourier space to satisfy the band limit, yielding a highly efficient code.
Their intuition is that the smoothing operator $\Riesz$ produces an effective bandlimit anyway.

The aim of this article is to prove that the above numerical space discretization from \cite{ZhFl19} converges if the initial data is regular enough:
We will show in \cref{thm:convergence} a more detailed version of the following.

\begin{theorem}[Convergence of bandlimited EPDiff equation]
Let $m>1+\frac d2$ and the metric on $\DiffSob^m$ be induced by a Fourier multiplier $\L$.
If the initial velocity $v_0$ has Sobolev regularity $m+k$ for $k\geq1$,
the numerical approximation of bandlimiting the right-hand sides in \eqref{eqn:EPDiffIntro} converges to the true solution as the bandlimit $R$ tends to infinity.
Moreover, for $k\geq2$ the error tends to zero at rate $R^{1-k}$.
\end{theorem}

(Note that numerically the convergence order seems to be better by one, so generically the convergence might be better.)
This shows that the earlier-mentioned intuition may be only partly correct: The smoothing operator $\Riesz$ alone might not produce a strong enough bandlimiting for the approximation to converge, but an additional, stronger smoothing of the initial condition might potentially be required.
(In fact, \cite{ZhFl19} additionally replaces differentiation by finite differences, but this only slightly simplifies the implementation.)
Following the theme that the underlying Lie group structure allows to use ODE techniques, the convergence is proven via Gronwall's inequality.
This is made possible by a regularity result for geodesics in $\DiffSob^m$,
sometimes known as ``no-loss-no-gain (of regularity)'' \cite[Cor.\,7.6]{BaHaMi25}, \cite[Sec.\,5.2]{Br17},
the analysis of which started with \cite{EbMa70}.
In \cref{thm:regularityPreservationLDDMM} we prove a version slightly stronger than existing ones, a short form of which is the following.

\begin{theorem}[Sobolev regularity preservation along geodesics]
Let $m>1+\frac d2$ and the metric on $\DiffSob^m$ be strong, twice differentiable, and induced by some $\L=B^*B+\bar\L$,
where $\bar\L$ is continuous from $H^m$ to $H^{1-m}$ and $H^m\times\DiffSob^m\ni(v,\phi)\mapsto[B(v\circ\phi^{-1})]\circ\phi\in L^2$ is differentiable.
If the initial velocity $v_0$ has Sobolev regularity $m+k$ for $k\geq1$,
then so does the velocity $v_t$ for all $t$.
\end{theorem}

This result builds upon the knowledge acquired during recent years on $\DiffSob^m$,
in particular the well-posedness of geodesics, the existence of shortest geodesics, and the rigorous derivation of the geodesic equation \cite{MiPr10,BrVi17,GuRaRuWi23}.
For smooth Riemannian metrics it can be derived from the right-invariance of the metric and the resulting Riemannian exponential map:
Indeed, via the right-invariance, (higher order) differentiability of the exponential can be traded in for preservation of (higher order) Sobolev exponents \cite[Cor.\,7.6]{BaHaMi25}, \cite[Sec.\,5.2]{Br17}.
However, that argument does not exploit that the derivatives of the exponential (such as $D\exp_\id(v)$) need only be bounded in very specific directions (such as $Dv$)
so that consequently the smoothness requirement on the metric can be weakened to the above.
Our result
confirms a conjecture of \cite[Rem.\,4.2]{MiPr10} who could only show preservation of Sobolev regularity $m+k$ for $k>m+\frac d2$ in space dimension $d$
(\cite[Thm.\,4.1]{MiPr10} is only stated for a particular class of smooth metrics, but the proof does not require the full smoothness).

Zhang and Fletcher viewed their numerical approximation in \cite{ZhFl19} as solving the EPDiff equation in a finite-dimensional \emph{approximate} Lie algebra
(in which the Lie bracket is replaced by a similar bilinear antisymmetric operation).
We will briefly discuss why an approximation via a Lie algebra of bandlimited functions with the original Lie bracket cannot exist.

Finally note that LDDMM and the EPDiff equation are also frequently employed with alternative smoothing operators $\Riesz$ such as convolution with a Gaussian.
None of the shown analysis applies to these settings, since almost nothing is so far known about the geometry of the associated group of diffeomorphisms:
Is it a manifold, and what (Banach) space $X$ is it modelled over?%
\footnote{Note that some progress has been made for groups of diffeomorphisms that are flows of vector fields from \emph{non-Hilbert} function spaces \cite{NeRa19}.}
Is the right-invariant Riemannian metric induced by $\L=\Riesz^{-1}$ smooth so that geodesics are well-defined?
Our numerical analysis further
requires an affirmative answer to the following additional questions:
Is $X$ a Banach algebra so that one can make sense of the quadratic terms in \eqref{eqn:EPDiffIntro}?
Does an estimate of the form $\|(\L v)Dv\|_{(X^+)'}\leq\|v\|_X^2$ hold with $X^+=\{v\in X\,|\,Dv\in X^d\}$ and $(X^+)'$ its dual space?
Do geodesics preserve the property $v_t\in X^+$ or even higher differentiability?

The outline of the article is as follows.
In \cref{sec:SobolevDiffeomorphisms} we recapitulate the known theory of Sobolev diffeomorphisms in order to introduce all necessary notions.
In \cref{sec:estimates} we prove estimates of operations with diffeomorphisms, in particular we prove the new regularity result for geodesics.
\Cref{sec:discretization} introduces the spectral space discretization and discusses principal obstructions to structure-preserving discretizations.
Finally, \cref{sec:convergence} proves the convergence of the discretization.


Below we briefly introduce some notation employed throughout.
We will work on the $d$-dimensional flat torus $\T^d$ (in applications $d\in\{2,3\}$),
which will be identified with $[0,1)^d$ with periodic boundary conditions.
The single stroke norms $|\cdot|$ and $|\cdot|_\infty$ denote the Euclidean $\ell^2$- and the $\ell^\infty$-norm on finite-dimensional vectors or tensors.
The identity matrix is $\onesymbol$, the identity operator on a function space is denoted by $\Id$, and the identity function by $\id$.
If we write $\id:\T^d\to\R^d$, then we identify $\T^d$ with $[0,1)^d$.
The adjoint of a linear operator $P$ is written as $P^*$.

We will employ the following function spaces:
$H^l(\T^d)$ denotes the Sobolev space of scalar functions on the torus with square-integrable weak derivatives up to order $l$.
For simplicity we will only work with integer orders; the special case $l=0$ refers to square-integrable Lebesgue functions.
For $\R^d$-valued functions we use the notation $H^l(\T^d;\R^d)$.
The notation $\|\cdot\|_{H^l}$ then refers to any one of the equivalent $H^l$-norms (for vector-valued functions simply taken componentwise).
In contrast, $\NHm{\cdot}$ corresponds to one particular $H^m$-norm (the one chosen when defining a metric on the space of Sobolev diffeomorphisms).
The corresponding dual spaces are denoted $H^{-l}(\T^d)$ and $H^{-l}(\T^d;\R^d)$, respectively.
Note that the diffeomorphisms will actually be functions in $H^m(\T^d;\T^d)$ (so domain and codomain are the torus);
this simply means that the diffeomorphism is in $H^m$ if restricted to any simply connected neighbourhood
and expressed in local coordinates of domain and co-domain (cf.\ \cite[\S\,4.3]{Ta23}; the local coordinates are obtained by the identification of $\T^d$ with $\R^d$ modulo the integer lattice).
The Sobolev space of functions with $l$ essentially bounded derivatives is denoted by $W^{l,\infty}$ (domain and codomain are indicated as for $H^l$) with norm $\|\cdot\|_{W^{l,\infty}}$,
where the special case $l=0$ indicates essentially bounded Lebesgue functions.
$C^n$ denotes $n$ times continuously differentiable functions (again domain and codomain will be indicated as for $H^l$), with $C^\infty$ representing infinitely often differentiable functions.
Finally, we will employ the Bochner space $L^1([0,1];H^m(\T^d;\R^d))$ with norm $\|\cdot\|_{L^1([0,1];H^m)}$ of absolutely integrable paths in $H^m(\T^d;\R^d)$.

For a function $w$ depending on time, a subscript $t$ as in $w_t$ denotes evaluation at time $t$, while a dot as in $\dot w_t$ denotes time differentiation.
The spatial derivative operator is denoted $D$, and we write $D\phi^{-1}$ for $D(\phi^{-1})$ with $\phi^{-1}$ the inverse of a diffeomorphism $\phi$ (as opposed to $(D\phi)^{-1}$).
We will use the Landau notations $f\in o(g)$ and $f\in O(g)$, and $f\lesssim g$ indicates that there is a constant $C>0$
(which may depend on $d$ and on the particular choices made for the involved norms) such that $f\leq Cg$.
Moreover, we will use $C$ as a generic constant that may change its value from line to line.

\section{The Riemannian manifold of Sobolev diffeomorphisms}\label{sec:SobolevDiffeomorphisms}
In this section we introduce all necessary notions and summarize what is known about the Riemannian manifold of Sobolev diffeomorphisms.
We consider diffeomorphisms on the $d$-dimensional flat torus $\T^d$, since the spectral discretization is devised for this setting.
However, statements analogous to the ones of this section also apply to other domains such as $\R^d$ or smooth compact manifolds (in the latter case one has to work with charts).

\begin{definition}[Group of Sobolev diffeomorphisms]
Let $m>\frac d2+1$.
The group of \emph{Sobolev diffeomorphisms} (on the torus) of Sobolev regularity $m$ is
\begin{equation*}
\DiffSob^m=\{\phi\in H^m(\T^d;\T^d)\,|\,\phi^{-1}\in H^m(\T^d;\T^d),\,\phi\text{ preserves orientation}\}
\end{equation*}
with group product $(\phi,\psi)\mapsto\phi\circ\psi$.
The subgroup formed by the connected component (in $H^m(\T^d;\T^d)$) of the identity is denoted $\DiffSobId^m$.
\end{definition}

That $\DiffSob^m$ actually forms a group follows from the regularity estimate
\begin{equation}\label{eqn:regularityComposition}
\|\psi\circ\phi\|_{H^s}\leq C\|\psi\|_{H^s}\qquad\text{for }0\leq s\leq m\text{ and all }\psi\in H^s(\T^d)
\end{equation}
with the constant $C$ depending on $\phi\in\DiffSob^m$ \cite[Lemma 2.7]{InKaTo13}.
Due to the continuous Sobolev embedding $H^m\hookrightarrow C^1$ for $m>\frac d2+1$, the group $\DiffSob^m$ is a subgroup of
\begin{equation*}
\Diff^1=\{\phi\in C^1(\T^d;\T^d)\,|\,\phi^{-1}\in C^1(\T^d;\T^d),\,\phi\text{ preserves orientation}\},
\end{equation*}
the group of $C^1$ diffeomorphisms of the torus.

We will only be concerned with $\DiffSobId^m$;
statements about other connected components $\mathcal C\subset\DiffSob^m$ (which do not form subgroups, though, for lack of the identity)
follow via the identification $\mathcal C=\phi\circ\DiffSobId^m$ for an arbitrary $\phi\in\mathcal C$.
At least in low dimensions, the decomposition into connected components is understood:
For $d=2$ the quotient of the group under identifying elements from the same connected component, the so-called mapping class group of $\T^2$,
is given by $\mathrm{SL}(2,\Z)$ \cite[Thm.\,2.5]{FaMa12} (in fact, this pertains from the group of homeomorphisms to the group of $C^\infty$-diffeomorphisms \cite[\S\,1.4]{FaMa12}).
For $d=3$ the mapping class group for homeomorphisms on $\T^d$ (and likely also for diffeomorphisms) is $\mathrm{SL}(3,\Z)$ \cite{Ha76};
for $d=4$ little seems to be known, and for $d\geq5$ the mapping class groups are more complicated and also depend on whether homeomorphisms or diffeomorphisms are considered \cite[Thm.\,4.1]{Ha78}.

It is well-known that $\DiffSobId^m$ is a differential ($C^\infty$) manifold modelled over the Hilbert space $H^m(\T^d;\R^d)$ \cite{MiPr10}.
In fact, this can be seen by interpreting $\DiffSobId^m$ as the quotient space $\DiffSobIdLift^m/\Z^d$ for
\begin{equation*}
\DiffSobIdLift^m=\{\phi:\T^d\to\R^d\,|\,\phi-\id\in H^m(\T^d;\R^d),\,\det D\phi>0\}
\end{equation*}
upon which the group $\Z^d$ acts by $\Z^d\times\DiffSobIdLift^m\ni(z,\phi)\mapsto T_z\phi\in\DiffSobIdLift^m$ with $(T_z\phi)(x)=\phi(x)+z$.
Essentially, taking the quotient just identifies the codomain $\R^d$ with $\T^d$ by taking the modulus with respect to one.
Now due to $H^m\hookrightarrow C^1$, the set $\DiffSobIdLift^m-\id$ is an open subset of $H^m(\T^d;\R^d)$ and thus a differential ($C^\infty$) manifold modelled over that space.
Therefore, the quotient $\DiffSobIdLift^m/\Z^d$ is a smooth submersion by the Quotient Manifold Theorem;
it is actually even a smooth covering map \cite[Thm.\,9.19]{Le03} so that $\DiffSobId^m$ is a differential manifold modelled over the same Hilbert space.
The submersion is given by the map
\begin{equation*}
\pi:\DiffSobIdLift^m\to\DiffSobId^m,\
\pi(\phi)=p\circ\phi
\quad\text{with}\quad
p:\R^d\to\T^d,\,p(x)=\left(\begin{smallmatrix}x_1\text{ mod }1\\\vdots\\x_d\text{ mod }1\end{smallmatrix}\right).
\end{equation*}
The tangent space to $\DiffSobId^m$ at any $\phi\in\DiffSobId^m$ is obviously given by
\begin{equation*}
T_\phi\DiffSobId^m=H^m(\T^d;\R^d).
\end{equation*}
From now on, we will denote by $\|\phi-\id\|_{H^m}$ for $\phi\in\DiffSobId^m$ the smallest $H^m$-norm along its fibre,
\begin{equation*}
\|\phi-\id\|_{H^m}=\min_{\psi\in\pi^{-1}(\phi)}\|\psi-\id\|_{H^m},
\end{equation*}
thus essentially we implicitly identify a Sobolev diffeomorphism $\phi$ with an element of $\DiffSobIdLift^m$.
Similarly, for any $\phi\in\DiffSobIdLift^m$ and $f\in H^s(\T^d)$ we will write
\begin{equation*}
f\circ\phi
\quad\text{for}\quad
f\circ \pi(\phi).
\qquad\text{and}\qquad
f\circ\phi^{-1}
\quad\text{for}\quad
f\circ \pi(\phi)^{-1}.
\end{equation*}

Since $\DiffSobId^m$ (and thus also $\DiffSob^m$) is a group as well as a manifold, it behaves in some aspects like a Lie group and is often formally treated like one.
However, inversion and left-multiplication are not smooth (not even differentiable), which is typically required for a Lie group
(such groups have been termed \emph{half-Lie groups} in \cite{KrMiRa15,MaNe18}).
In fact, denoting the group of homeomorphisms and $C^\infty$-diffeomorphisms on the torus by $\Hom$ and $\Diff^\infty$, respectively,
by \cite[III, Exerc.\,\S4, \P7]{Bo72} there cannot exist a Banach Lie group $G$
such that $\Diff^\infty\subset G\subset\Hom$ and the inclusions $\Diff^\infty\hookrightarrow G\hookrightarrow\Hom$ are continuous homomorphisms.
In more detail, one can construct diffeomorphisms $\phi\in\Diff^\infty$ that are arbitrarily close to $\id$, but do not belong to any one-parameter subgroup of $\Diff^\infty$;
thus they do not belong to a one-parameter subgroup of $G$ either.
This, however, contradicts the local surjectivity of the Lie exponential map in Banach Lie groups.

The manifold $\DiffSobId^m$ is equipped with a right-invariant Riemannian metric.

\begin{definition}[Right-invariant Sobolev metric]\label{def:SobolevMetric}
Let $\IPHm{\cdot}{\cdot}$ be an inner product of $H^m(\T^d;\R^d)$ with induced norm $\NHm{\cdot}$ equivalent to the standard $H^m$-norm, then
\begin{equation}\label{eqn:metric}
g_\phi(v,w)=\IPHm{v\circ\phi^{-1}}{w\circ\phi^{-1}}
\qquad\text{for }
\phi\in\DiffSobId^m,\,v,w\in T_\phi\DiffSobId^m
\end{equation}
is the induced \emph{right-invariant Riemannian Sobolev metric} on $\DiffSobId^m$.\footnote{%
Note that in contrast to most standard textbooks on Riemannian geometry
we do not impose any smoothness requirements on a Riemannian metric, but simply take it to be a collection of inner products -- one on each tangent space to the Riemannian manifold.
Depending on the choice of $\IPHm\cdot\cdot$, \eqref{eqn:metric} will of course entail \emph{some} regularity of the map $\phi\mapsto g_\phi$, but this is ignored for now.%
}
\end{definition}

The right-invariance $g_{\phi\circ\psi}(v\circ\psi,w\circ\psi)=g_\phi(v,w)$ follows from the definition, the well-definedness of the metric follows from \eqref{eqn:regularityComposition}.
The metric is strong in the sense that it induces the topology of the tangent space;
sometimes also weak metrics (inducing a weaker topology) are considered in the literature, which we will come back to in \cref{rem:weakMetrics}.
We will denote the Riesz isomorphism associated with $\IPHm{\cdot}{\cdot}$ by
\begin{equation*}
\Riesz:H^{-m}(\T^d;\R^d)\to H^m(\T^d;\R^d)
\qquad\text{with inverse }
\L=\Riesz^{-1}.
\end{equation*}
It is defined via the relation $\langle v,w\rangle=\IPHm{\Riesz v}{w}$ for all $v\in H^{-m}(\T^d;\R^d)$ and $w\in H^m(\T^d;\R^d)$,
and $\L$ typically is a differential operator, e.g.\ $\L=(1-\Delta)^m$.

Any Riemannian metric induces a path energy and Riemannian distance.

\begin{definition}[Path energy and Riemannian distance on $\DiffSobId^m$]
The \emph{Riemannian distance} $\dist$ on $\DiffSobId^m$ induced by the right-invariant Riemannian Sobolev metric is given by
\begin{equation*}
\dist(\chi,\psi)^2
=\inf_{\substack{\phi:[0,1]\to\DiffSobId^m\\\phi_0=\chi,\,\phi_1=\psi}}E[\phi]
\end{equation*}
for the \emph{path energy}
\begin{equation*}
E[\phi]=\int_0^1g_{\phi_t}(\dot\phi_t,\dot\phi_t)\,\d t=\int_0^1\IPHm{\dot\phi_t\circ\phi_t^{-1}}{\dot\phi_t\circ\phi_t^{-1}}\,\d t.
\end{equation*}
\end{definition}

It is a standard argument that the path energy $E$ is an upper bound for the squared path length (with equality on paths that have unit speed) \cite[Prop.\,1.8.7]{Kl95a},
which is why $\dist$ is a length metric.
Shortest connecting paths, i.e.\ shortest geodesics between two elements $\phi_0,\phi_1\in\DiffSobId^m$, are paths $\phi:[0,1]\to\DiffSobId^m$ that minimize $E[\phi]$
among all paths with same end points $\phi_0,\phi_1$.
Obviously, an equivalent characterization of geodesics is as solutions of the constrained problem
\begin{equation*}
\min_{\substack{\phi:[0,1]\to\DiffSobId^m\\v:[0,1]\to H^m(\T^d;\R^d)}}
\int_0^1\IPHm{v_t}{v_t}\,\d t
\quad\text{such that }\dot\phi_t=v_t\circ\phi_t.
\end{equation*}

\begin{definition}[Eulerian velocity and flow]
Let $\phi:[0,1]\to\DiffSobId^m$, $v:[0,1]\to H^m(\T^d;\R^d)$ satisfy
\begin{equation}\label{eqn:flow}
\dot\phi_t=v_t\circ\phi_t,
\end{equation}
then the (time-dependent) vector field $v$ is the \emph{Eulerian velocity} field of the path $\phi$, while $\phi$ is the \emph{flow} of $v$.
\end{definition}

The geodesic equation is the optimality condition for the minimization of $E$ and thus a second order ODE on the manifold $\DiffSobId^m$ (which is a PDE on $[0,1]\times\T^d$).
On honest Lie groups with right-invariant metric, the geodesic ODE is invariant under right action of the group on itself so that by Noether's Theorem it can be reduced to a first order ODE.
One can apply the same argument to formally derive such a first order geodesic equation on $\DiffSobId^m$:
Let $\phi$ be a geodesic path with velocity $v$ and let $\eta\in C^\infty([0,1]\times\T^d)$ be an infinitesimal perturbation with $\eta_0=\eta_1=0$.
Defining $\phi^\varepsilon_t=(\id+\varepsilon\eta_t)\circ\phi_t$,
its velocity is given by $v^\varepsilon_t=\dot\phi^\varepsilon_t\circ(\phi^\varepsilon_t)^{-1}=\left[(\onesymbol+\varepsilon D\eta_t)v_t+\varepsilon\dot\eta_t\right]\circ(\id+\varepsilon\eta_t)^{-1}$.
Assuming sufficient differentiability, the optimality conditions read
\begin{multline}\label{eqn:optimalityConditions}
0=\left.\frac\d{\d\varepsilon}E[\phi^\varepsilon]\right|_{\varepsilon=0}
=\left.\frac\d{\d\varepsilon}\int_0^1\IPHm{v^\varepsilon_t}{v^\varepsilon_t}\,\d t\right|_{\varepsilon=0}
=\left.2\int_0^1\langle\L v^\varepsilon_t,\tfrac{\d v^\varepsilon_t}{\d\varepsilon}\rangle\,\d t\right|_{\varepsilon=0}\\
=2\int_0^1\langle\L v_t,D\eta_tv_t+\dot\eta_t-Dv_t\eta_t\rangle\,\d t
=2\int_0^1\langle\L v_t,[v_t,\eta_t]+\dot\eta_t\rangle\,\d t
\end{multline}
for the Lie bracket
\begin{equation*}
[v,w]=(Dw)v-(Dv)w.
\end{equation*}
The Lie bracket is usually written in terms of the adjoint representation $\adjointDer_vw$:
Differentiating the conjugation isomorphism $\iota_\phi:\psi\mapsto\phi\circ\psi\circ\phi^{-1}$ at $\psi=\id$ yields the so-called adjoint map
\begin{equation*}
\adjointMap_\phi=\partial_\psi\iota_\phi|_{\psi=\id}
\qquad\text{with }
\adjointMap_\phi w=(D\phi\,w)\circ\phi^{-1},
\end{equation*}
whose derivative with respect to $\phi$ is known as the (Lie algebra) adjoint representation
\begin{equation*}
\adjointDer_vw=\partial_\phi\adjointMap_\phi(v)w|_{\phi=\id}=-[v,w].
\end{equation*}
Thus, after an integration by parts and applying the fundamental lemma of the calculus of variations to \eqref{eqn:optimalityConditions} we arrive at the formal geodesic ODE
\begin{equation}\label{eqn:EPDiff}
\dot\rho_t
=-\adjointDer_{v_t}^*\rho_t
=-(\div(\rho_t\otimes v_t)+(Dv_t)^T\rho_t)
\qquad\text{for the momentum }
\rho_t=\L v_t.
\end{equation}
This first order ODE is known as the EPDiff or Euler--Poincar\'e equation \cite{TrYo15};
its integral yields $\rho$ or equivalently $v$, from which then the geodesic path $\phi$ can be calculated as the flow.
Note that the derivation was merely formal since along the way we illegally differentiated the group product of $\DiffSobId^m$
(e.g.\ $\adjointMap_\phi$ with $\phi\in\DiffSobId^m$ maps $H^m(\T^d;\R^d)$ only into $H^{m-1}(\T^d;\R^d)$).
As a consequence, due to $v_t\in H^m(\T^d;\R^d)$ and $\rho_t\in H^{-m}(\T^d;\R^d)$, the right-hand side of \eqref{eqn:EPDiff} only lies in $H^{-m-1}(\T^d;\R^d)$, and the equation may not be well-defined.

With these preliminaries it turns out that for $m>\frac d2+1$
\begin{enumerate}
\item
$\DiffSobId^m$ is a topological group (this can be checked directly, exploiting that the constant in \eqref{eqn:regularityComposition} only depends on $\min_x\det D\phi(x)$ and $\|\phi-\id\|_{H^m}$),
\item
the flow $\phi$ (starting from $\phi_0=\id$) of any velocity $v\in L^1([0,1];H^m(\T^d;\R^d))$ lies in $\DiffSobId^m$ \cite{BrVi17},
\item
paths of finite energy in $\DiffSobId^m$ exist between any two elements of $\DiffSobId^m$
(since $\DiffSob^m$ is open and locally connected in $H^m(\T^d;\R^d)$, its connected components, in particular $\DiffSobId^m$, are open; 
now connected open subsets of normed spaces are polygonally connected, 
and a polygon in $\DiffSobId^m$, consisting of linear interpolation segments $t\mapsto t\phi_1+(1-t)\phi_0$, has finite path energy),
\item
shortest geodesics in $\DiffSobId^m$ exist between any two elements of $\DiffSobId^m$ \cite{Tr95a}
(Trouv\'e showed this for the set of diffeomorphisms reachable via a path of finite energy, which by the previous two points turns out to be $\DiffSobId^m$),
\item
$\DiffSobId^m$ with Riemannian distance $\dist$ is complete as a metric space \cite{Tr95a},
\item
for the inner product $\IPHm{v}{w}=\langle\L v,w\rangle$ with $\L=(1-\Delta)^m$
or a more general differential operator of order $2m$ with smooth coefficients the metric is smooth \cite[Thm.\,4.1]{MiPr10} (this even holds for any elliptic invertible operator $\L\in\mathrm{OPS}_{1,0}^{2m}$ of order $2m$ \cite[Thm.\,4.16 \& Thm.\,5.1]{BaBrCiEsKo20}, see \cite{EsKo14,BaEsKo15} for the Fourier multiplier case),
thus the geodesic ODE is locally uniquely solvable by classical differential geometry and ODE theory \cite[\S\,1.6]{Kl95a},
its solution depends smoothly on the initial conditions,
and by the Inverse Function Theorem the Riemannian exponential map is a local diffeomorphism,
\item
for the same inner product, $\DiffSobId^m$ is geodesically complete, i.e.\ geodesics can be extended for all times
(this is a consequence of the previous two points by \cite[Prop.\,6.5]{La95}),
\item
for an inner product $\IPHm{v}{w}=\langle\L v,w\rangle$ with $\L=B^*B+\bar\L$
for $B$ a differential operator of order $m$ with bounded coefficients and compact self-adjoint $\bar\L:H^m(\T^d;\R^d)\to H^{-m}(\T^d;\R^d)$, continuous into $H^{-m+1}(\T^d;\R^d)$,
geodesics satisfy a weak PDE, whose strong form is \eqref{eqn:EPDiff} \cite{GuRaRuWi23} (cf.\ \cref{thm:GuRaRuWi23}).
\end{enumerate}

In summary, with an appropriate inner product $\IPHm{\cdot}{\cdot}$, the group $\DiffSobId^m$ is a metrically and geodesically complete Riemannian manifold
in which shortest geodesics between any two elements exist, which satisfy a weak PDE, whose strong form is \eqref{eqn:EPDiff}.

\section{Refined and new regularity estimates}\label{sec:estimates}
In this section we give a regularity result for geodesics, \cref{thm:regularityPreservationLDDMM}, that allows to prove convergence of the spectral space discretization.
Furthermore, in order to be able to provide the explicit dependence of the final error estimates on the initial condition $v_0$ of the geodesic,
we revisit and refine in \cref{thm:Regulcomposition,thm:normEstDeform} a few known regularity estimates for operations with Sobolev diffeomorphisms.
For completeness we also provide in \cref{thm:smoothMetric} a sufficient condition for $l$-fold differentiability of the Riemannian metric on $\DiffSobId^m$,
since our convergence results require two- and threefold differentiability.
Again the results and proofs can readily be adapted from the torus $\T^d$ to other domains such as $\R^d$ or smooth compact manifolds.

We first recall a continuity result from \cite{BeHo21} (there given for general Lipschitz domains and integrability exponents) for pointwise multiplication of Sobolev functions that we will frequently make use of.
\begin{lemma}[{Continuity of multiplication, \cite[Cor.\,6.3, Thm.\,7.4]{BeHo21}}]\label{thm:BeHo21}
Let $0\leq s\leq r,m$ with $r+m>s+\frac d2$, then pointwise multiplication is a continuous bilinear map $H^m(\T^d)\times H^r(\T^d)\to H^s(\T^d)$.
\end{lemma}

We begin with an estimate for the composition with a small deformation (iterations of which will then yield estimates for general diffeomorphisms),
which refines the following statement from \cite{InKaTo13} (given there on $\R^d$ instead of $\T^d$, but the proof carries over).
\begin{lemma}[{Continuity of composition, \cite[Lemma 2.7]{InKaTo13}}]\label{thm:InKaTo13}
Let $m>\frac d2+1$ and $0\leq s\leq m$.
Then $H^s(\T^d)\times\DiffSob^m\ni(f,\phi)\mapsto f\circ\phi\in H^s(\T^d)$ is continuous.
Moreover, for any $m,M>0$ there exists $C>0$ such that $\det D\phi\geq m$ and $\|\phi-\id\|_{H^m}\leq M$ for $\phi\in\DiffSob^m$ imply
\begin{equation*}
\|f\circ\phi\|_{H^s}\leq C\|f\|_{H^s}
\qquad\text{for all }
f\in H^s(\T^d).
\end{equation*}
\end{lemma}
Our refinement consists in the precise control of the involved constant in terms of how close the diffeomorphism is to the identity.
Below we will abbreviate the closed norm ball in $H^m(\T^d;\R^d)$ of radius $\varepsilon>0$ by
\begin{equation*}
\overline{B_\varepsilon(0)}=\{\psi\in H^m(\T^d;\R^d)\,|\,\|\psi\|_{H^m}\leq\varepsilon\}.
\end{equation*}

\begin{lemma}[Regularity of composition]\label{thm:Regulcomposition}
Let $m>\frac d2+1$ and $0\leq s\leq m$.
Let $\varepsilon>0$ small enough such that $\id+\overline{B_\varepsilon(0)}\subset\DiffSobIdLift^m$. There exists $C>0$ such that
\begin{equation*}
\|f\circ\phi\|_{H^s}\leq(1+C\|\phi-\id\|_{H^m})\|f\|_{H^s}
\qquad\text{for all }
\phi\in\id+\overline{B_\varepsilon(0)},\ f\in H^s(\T^d).
\end{equation*}
\end{lemma}
\begin{proof}
First note that by Sobolev embedding we have $\|\phi-\id\|_{C^1}\leq c\|\phi-\id\|_{H^m}$ for some constant $c>0$ (which of course depends on the chosen inner product on $H^m(\T^d;\R^d)$).
Thus, in particular $\|D\phi-\onesymbol\|_{C^0}\leq c\varepsilon$ for all $\phi\in\id+\overline{B_{\varepsilon}(0)}$,
and $\varepsilon$ can indeed be chosen small enough such that $\det D\phi>0$ everywhere and therefore $\id+\overline{B_{\varepsilon}(0)}\subset\DiffSobIdLift^m$.

The proof is by induction in $s$.
Without loss of generality we consider the standard norm $\|f\|_{H^s}^2=\sum_{j=0}^s\int_{\T^d}|D^jf|^2\,\d x$.
For $s=0$ we have
\begin{multline*}
\|f\circ\phi\|_{H^0}^2
=\int_{\T^d}|f\circ\phi|^2\,\d x
=\int_{\T^d}|f|^2\det D(\phi^{-1})\,\d x\\
\leq\min_{x\in\T^d}\frac1{\det D\phi(x)}\|f\|_{H^0}^2
\leq[(1+cL\|\phi-\id\|_{H^m})\|f\|_{H^0}]^2,
\end{multline*}
where $L$ is the Lipschitz constant of the smooth map $A\mapsto\sqrt{1/\det A}$ on the closed ball of radius $c\varepsilon$ around the identity matrix $\onesymbol$.
Now let $s>0$.
We have $D(f\circ\phi)=Df\circ\phi\, D\phi$.
By the induction hypothesis $\|Df\circ\phi\|_{H^{s-1}}\leq(1+\kappa\|\phi-\id\|_{H^m})\|Df\|_{H^{s-1}}$ for some $\kappa>0$.
Furthermore,
\begin{multline*}
\|Df\circ\phi\, D\phi\|_{H^{s-1}}
\leq\|Df\circ\phi\|_{H^{s-1}}+\|Df\circ\phi(D\phi-\onesymbol)\|_{H^{s-1}}\\
\leq\|Df\circ\phi\|_{H^{s-1}}(1+\hat K\|D\phi-\onesymbol\|_{H^{m-1}})
\leq\|Df\circ\phi\|_{H^{s-1}}(1+K\|\phi-\id\|_{H^m})
\end{multline*}
for some $\hat K,K>0$ by \cite[Lemma 2.3]{InKaTo13} or \cref{thm:BeHo21} (essentially since $H^{m-1}(\T^d)$ is a Banach algebra).
In summary,
\begin{align*}
&\,\quad\|f\circ\phi\|_{H^s}^2\\
&=\|f\circ\phi\|_{H^0}^2+\|D(f\circ\phi)\|_{H^{s-1}}^2\\
&\leq(1+cL\|\phi-\id\|_{H^m})^2\|f\|_{H^0}^2+(1+K\|\phi-\id\|_{H^m})^2\|Df\circ\phi\|_{H^{s-1}}^2\\
&\leq(\!1\!+\!cL\|\phi\!-\!\id\|_{H^m}\!)^2\|f\|_{H^0}^2\!+\!(\!1\!+\!K\|\phi\!-\!\id\|_{H^m}\!)^2(\!1\!+\!\kappa\|\phi\!-\!\id\|_{H^m}\!)^2\|Df\|_{H^{s\!-\!1}}^2\\
&\leq(1+C\|\phi-\id\|_{H^m})^2\|f\|_{H^s}^2
\end{align*}
for some $C>0$ (depending on $c,L,K,\kappa$ and $\varepsilon$).
\end{proof}

Based on \cref{thm:Regulcomposition}, we next estimate the $H^m$-norm of the flow, 
generalize \cref{thm:Regulcomposition} to metric balls, and similarly estimate the adjoint map associated with the flow.
The proof mainly is a variation of \cite[Lemma 3.5]{BrVi17} to obtain the explicit dependence on the metric distance to the identity
and to estimate in addition the adjoint map (since it will carry regularity information along geodesics).
Given the structure of the flow \eqref{eqn:flow}, it is not surprising that the estimates depend exponentially on the metric distance to the identity.

\begin{proposition}[Norm and composition estimates for the flow]\label{thm:normEstDeform}
Let $m>\frac{d}{2}+1$, $0\leq s\leq m$, and $I=[0,1]$.
There exists a constant $C>0$ such that for all $v\in L^1(I;H^m(\T^d;\R^d))$ with flow $\phi:I\to\DiffSobId^m$, $\phi_0=\id$, it holds
\begin{enumerate}
\item
$\|\phi_t-\id\|_{H^m}\leq\exp(C\|v\|_{L^1(I;H^m)})-1$ for all $t\in I$,
\item\label{enm:composition}
$\|f\circ\phi_t\|_{H^s}\leq\exp(C\|v\|_{L^1(I;H^m)})\|f\|_{H^s}$ for all $t\in I$, $f\in H^s(\T^d)$,
\item\label{enm:adjointMapEstimate}
if $s\leq m-1$, then $\|\adjointMap_{\phi_t}w\|_{H^s}\leq\exp(C\|v\|_{L^1(I;H^m)})\|w\|_{H^s}$ for all $t\in I$, $w\in H^s(\T^d;\R^d)$.
\end{enumerate}
\end{proposition}
\begin{proof}
First note that 
we may consider the flow $\phi_t$ in $\DiffSobIdLift^m$ instead of $\DiffSobId^m$ since (with our notation conventions from the previous section) this does not change any of the estimates.

Next consider $v$ with $(1+C\varepsilon)\|v\|_{L^1(I;H^m)}<\varepsilon$ for $\varepsilon$ and $C$ from \cref{thm:Regulcomposition}.
We show $\|\phi_t-\id\|_{H^m}\leq\varepsilon$ for all $t\in I$.
Indeed, let $T\in I$ be the smallest time such that either $\|\phi_T-\id\|_{H^m}=\varepsilon$ or $T=1$, then for $t<T$ we have
\begin{multline}\label{eqn:distFromId}
\|\phi_t-\id\|_{H^m}
\leq\int_0^t\|v_r\circ\phi_r\|_{H^m}\d r\\
\leq(1+C\varepsilon)\int_I\|v_r\|_{H^m}\d r
=(1+C\varepsilon)\|v\|_{L^1(I;H^m)}<\varepsilon.
\end{multline}
The last inequality remains strict even after taking the liminf as $t\to T$,
while $\liminf_{t\to T}\|\phi_t-\id\|_{H^m}$ can be bounded below by $\|\phi_T-\id\|_{H^m}$
(due to the weak lower semi-continuity of the norm and the uniform convergence $\phi_t\to\phi_T$, which implies $\phi_t-\id\to\phi_T-\id$ weakly in $H^m$ due to $\phi_T-\id\in H^m$).
Thus we must have $T=1$ and therefore $\|\phi_t-\id\|_{H^m}\leq\varepsilon$ for $t\leq1$.

From now on, $v$ is an arbitrary velocity field.
We split it into $N$ time segments $v^j=v|_{[t_{j-1},t_{j}]}$ (extended to $I$ by zero)
with $0=t_0<\ldots<t_N=1$ chosen such that $\int_{t_{j-1}}^{t_{j}}\|v_t\|_{H^m}\d t=\|v\|_{L^1(I;H^m)}/N<\frac\varepsilon{1+C\varepsilon}$.
Let us denote by $\phi^j$ the flow of $v^j$, then by the above we have $\|\phi^j_t-\id\|_{H^m}\leq\varepsilon$, and furthermore
\begin{equation*}
\phi_t=\phi^N_t\circ\cdots\circ\phi^2_t\circ\phi^1_t.
\end{equation*}

\begin{enumerate}
\item
Abbreviate $V_N\!=\!(1\!+\!C\varepsilon)\|v\|_{L^1(I;H^m)}/N$.
From \eqref{eqn:distFromId} we know $\|\phi^j_t-\id\|_{H^m}\leq V_N$ for all $j$.
By induction in $n$ we have $\|\phi^{j+n-1}_t\circ\cdots\circ\phi^j_t-\id\|_{H^m}\leq(1+CV_N)^n-1$ for all $j$ (assuming without loss of generality $C\geq1$):
Indeed, the case $n=1$ is trivial, and for $n>1$, using \cref{thm:Regulcomposition} we have
\begin{align*}
&\|\phi^{j+n-1}_t\circ\cdots\circ\phi^j_t-\id\|_{H^m}\\
\leq&\|(\phi^{j+n-1}_t\circ\cdots\circ\phi^{j+1}_t-\id)\circ\phi^j_t\|_{H^m}+\|\phi^j_t-\id\|_{H^m}\\
\leq&(1+CV_N)\|\phi^{j+n-1}_t\circ\cdots\circ\phi^{j+1}_t-\id\|_{H^m}+\|\phi^j_t-\id\|_{H^m}\\
\leq&(1+CV_N)[(1+CV_N)^{n-1}-1]+V_N
\leq(1+CV_N)^n-1.
\end{align*}
Thus, $\|\phi_t-\id\|_{H^m}\leq(1+CV_N)^N\to_{N\to\infty}\exp(C(1+C\varepsilon)\|v\|_{L^1(I;H^m)})$.
\item
From $\|f\circ\phi_t^N\circ\cdots\circ\phi_t^1\|_{H^s}=\|(\cdots(f\circ\phi_t^N)\circ\cdots\circ\phi_t^2)\circ\phi_t^1\|_{H^s}$ it follows
by inductively applying \cref{thm:Regulcomposition} that
$\|f\circ\phi_t\|_{H^s}\leq\|f\|_{H^s}\prod_{j=1}^N(1+C\|\phi_t^j-\id\|_{H^m})\leq\|f\|_{H^s}(1+CV_N)^N\to_{N\to\infty}\|f\|_{H^s}\exp(C(1+C\varepsilon)\|v\|_{L^1(I;H^m)})$.
\item
We will show $\|\adjointMap_{\phi_t^j}w\|_{H^s}\leq(1+\kappa V_N)\|w\|_{H^s}$ for all $j$ and some constant $\kappa>0$,
from which the result will follow exploiting $\adjointMap_{\phi_t}w=\adjointMap_{\phi_t^N}\cdots\adjointMap_{\phi_t^1}w$ and letting $N\to\infty$ as before. Indeed,
\begin{align*}
\|\adjointMap_{\phi_t^j}w\|_{H^s}
&\leq(1+C\|(\phi_t^j)^{-1}-\id\|_{H^m})\|D\phi_t^jw\|_{H^s}\\
&=(1+C\|(\phi_t^j)^{-1}-\id\|_{H^m})\|(D\phi_t^j-\onesymbol)w+w\|_{H^s}\\
&\leq(1+C\|(\phi_t^j)^{-1}-\id\|_{H^m})(c\|D\phi_t^j-\onesymbol\|_{H^m}+1)\|w\|_{H^s},
\end{align*}
using \cref{thm:BeHo21} with constant $c>0$ in the last inequality and \cref{thm:Regulcomposition} in the first.
The latter is allowed since $(\phi_t^j)^{-1}$ is the flow of $[0,t]\ni r\mapsto-v^j(t-r)$ \cite{Tr95a} so that by \eqref{eqn:distFromId} we also have $\|(\phi_t^j)^{-1}-\id\|_{H^m}\leq V_N<\varepsilon$.
Consequently, with some $\tilde c>0$ we have $\|\adjointMap_{\phi_t^j}w\|_{H^s}\leq(1+CV_N)(1+\tilde cV_N)\|w\|_{H^s}$, as desired.
\qedhere
\end{enumerate}
\end{proof}

We will also need certain continuity properties of composition and adjoint map
(note that continuity of $H^s(\T^d;\R^d)\times\DiffSobId^m\ni(w,\phi)\mapsto(w\circ\phi,\adjointMap_\phi w)\in H^s(\T^d;\R^d)\times H^s(\T^d;\R^d)$ for $m>\frac d2+1$ and $0\leq s<m$ is known, see \cref{thm:InKaTo13}).

\begin{lemma}[Continuity of $\adjointMap^*$]\label{thm:continuityAdjointAdjointMap}
Let $m>\frac d2+1$, $0\leq s<m$, and $\phi_n\to\phi$ in $\DiffSobId^m$ as $n\to\infty$.
Then $w_n\to w$ weakly in $H^s(\T^d;\R^d)$ implies
$w_n\circ\phi_n\to w\circ\phi$ and $\adjointMap_{\phi_n}w_n\to\adjointMap_\phi w$ both weakly in $H^s(\T^d;\R^d)$.
As a consequence, the map $(\rho,\phi)\mapsto\adjointMap_\phi^*\rho$ is continuous from $H^{-s}(\T^d;\R^d)\times\DiffSobId^m$ into $H^{-s}(\T^d;\R^d)$.
\end{lemma}
\begin{proof}
By \cref{thm:normEstDeform}\eqref{enm:composition} $w_n\circ\phi_n$ is bounded, hence contains a weakly converging subsequence (for simplicity still indexed by $n$).
We show that the weak limit is $w\circ\phi$ independent of the chosen subsequence so that actually the whole sequence converges weakly.
To this end let $\eta$ be smooth or at least $H^m$-regular, then
\begin{equation*}
\langle\eta,w_n\circ\phi_n-w\circ\phi\rangle
=\langle\eta,(w_n-w)\circ\phi_n+w\circ\phi_n-w\circ\phi\rangle
\mathop\to_{n\to\infty}\lim_{n\to\infty}\langle\eta,(w_n-w)\circ\phi_n\rangle
\end{equation*}
due to the continuity of the composition. Next, by the transformation rule we have
\begin{equation*}
\langle\eta,(w_n-w)\circ\phi_n\rangle
=\int_{\T^d}\det D\phi_n^{-1}\,\eta\circ\phi_n^{-1}\cdot(w_n-w)\,\d x
\mathop\to_{n\to\infty}0
\end{equation*}
since $\eta\circ\phi_n^{-1}\to\eta\circ\phi^{-1}$ strongly in $H^m(\T^d;\R^d)$ (by continuity of inversion in $\DiffSobId^m$ and composition with $\DiffSobId^m$ \cite[Lemmas 2.7-2.8]{InKaTo13}),
$\det D\phi_n^{-1}\to\det D\phi^{-1}$ strongly in $H^{m-1}(\T^d)$
(since $\phi\mapsto\phi^{-1}$ is continuous on $\DiffSobId^m$,
$\psi\mapsto D\psi$ is continuous from $\DiffSobId^m$ to $H^{m-1}(\T^d;\R^{d\times d})$,
and taking the pointwise determinant is continuous from $H^{m-1}(\T^d;\R^{d\times d})$ into $H^{m-1}(\T^d)$ since $H^{m-1}$ forms a Banach algebra \cite[Lemma 2.16]{InKaTo13}),
and thus $\det D\phi_n^{-1}\,\eta\circ\phi_n^{-1}\to\det D\phi^{-1}\,\eta\circ\phi^{-1}$ strongly in $H^{m-1}(\T^d;\R^d)$ by \cref{thm:BeHo21}.

Likewise, $D\phi_nw_n$ converges weakly to $D\phi w$ in $H^s(\T^d;\R^d)$ as the product of a strongly (in $H^{m-1}$) and a weakly (in $H^s$) converging sequence
so that $\adjointMap_{\phi_n}w_n$ converges weakly to $\adjointMap_\phi w$ by $\phi_n^{-1}\to\phi^{-1}$ in $\DiffSobId^m$ and the previous.

Finally, suppose $\adjointMap_\phi^*\rho$ were not continuous in $\phi$ and $\rho$.
Then there exist $C>0$ as well as sequences $\phi_n\to\phi$ in $\DiffSobId^m$, $\rho_n\to\rho$ in $H^{-s}(\T^d;\R^d)$ and  $w_n$ in $H^s(\T^d;\R^d)$ with $\|w_n\|_{H^s}\leq 1$ such that
\begin{equation*}
C
<\langle\adjointMap_{\phi_n}^*\rho_n-\adjointMap_{\phi}^*\rho,w_n\rangle
=\langle\rho_n,\adjointMap_{\phi_n}w_n\rangle-\langle\rho,\adjointMap_{\phi}w_n\rangle.
\end{equation*}
However, due to the boundedness of $w_n$ we can extract a subsequence (for simplicity again indexed by $n$) with $w_n\to w$ weakly in $H^s(\T^d;\R^d)$.
Then along this subsequence the right-hand side of the above converges to $\langle\rho,\adjointMap_{\phi}w\rangle-\langle\rho,\adjointMap_{\phi}w\rangle=0$, yielding a contradiction.
\end{proof}

Before proving a regularity estimate for geodesics (which then will allow an approximation via spectral discretization),
we need to ensure well-posedness of geodesics in the first place.
Even though shortest geodesics exist between any two points of $\DiffSobId^m$
\cite[Thm.\,3.7]{Tr95a},\cite[Thm.\,7.2]{BrVi17}\footnote{Note that \cite[Thm.\,7.2]{BrVi17} only states the result for smooth metrics,
but does not use the smoothness in proving existence of path energy minimizers.}
and under certain conditions on the inner product $\IPHm{\cdot}{\cdot}$ even satisfy a (weak) geodesic equation \cite{GuRaRuWi23} (cf.\ \cref{thm:GuRaRuWi23}),
the solution of this geodesic equation forward in time may not be unique.
To this end we appeal to classical differential geometry:
If the Riemannian metric $g_\phi(v,w)=\langle\L(v\circ\phi),w\circ\phi\rangle$ is twice continuously differentiable in $\phi$,
then the right-hand side of the second order geodesic equation is Lipschitz (even differentiable) and thus locally uniquely solvable \cite[\S\,1.6]{Kl95a}
with the solution depending continuously on the initial condition by classical ODE theory.
In other words, the Riemannian exponential map is locally well-defined and continuous (and $n$ times differentiable if the metric has $n$ further derivatives).
For the sake of completeness we provide a sufficient condition for this differentiability requirement, which essentially is a variation of \cite[Thm.\,4.1]{MiPr10}.

\begin{lemma}[Differentiable Sobolev metrics]\label{thm:smoothMetric}
Let $m>\frac d2+1$, $0\leq l\leq m$.
If $\L=\bar\L+B^*B$ with $\bar\L:H^{m-l}(\T^d;\R^d)\to H^{l-m}(\T^d;\R^d)$ bounded linear
and $B=\sum_{k=0}^{m}a_kD^k$ a differential operator of order $m$ with coefficient tensors $a_i$ of $W^{l,\infty}$-regularity,
then the metric \eqref{eqn:metric} is $l$ times continuously differentiable.
\end{lemma}
\begin{remark}[Example metrics]\label{rem:exampleMetrics}
The result of course applies to the standard example $\L=1+(-\Delta)^m$ with the choice $\bar\L=0$ and $B=(1,\Delta^{m/2})$ (for even $m$) or $(1,\nabla\Delta^{(m-1)/2})$ (for odd $m$),
but one may just as well introduce non-local operators or spatially inhomogeneous weights,
e.g.\ by choosing $\bar\L u(x)=\int_{\T^d}\frac{u(x)-u(y)}{\mathrm{dist}(x,y)^{d+2s}}\omega(x,y)\,\d y$ for some $s\in(0,1)\cap(0,m-l)$ and a bounded weight $\omega(x,y)$
or by choosing $\bar\L u=-\div(\chi_A\nabla u)$ for the characteristic function $\chi_A$ of some measurable $A\subset\T^d$ (as long as $l\leq m-1$).
In particular the latter choice may be of interest in applications, modelling an enhanced energy dissipation (which is one possible physical interpretation of the path energy) in region $A$,
e.g.\ due to a low permeability porous structure.
The resulting metric is not smooth, but only $m-1$ times continuously differentiable (as one can for instance easily verify for the concrete choice $d=1,m=2,A=[0,\frac12]$).
\end{remark}

The conditions of \cref{thm:smoothMetric} are clearly only sufficient:
For instance, adding a Fourier multiplier of order $m$ or even an operator from $\mathrm{OPS}_{1,0}^{m}$ to $B$ does not change the differentiability properties of the metric,
since these induce smooth metrics by \cite{BaEsKo15,BaBrCiEsKo20}.

\begin{proof}
We check differentiability in the global chart $\DiffSobIdLift^m\subset\id+H^m(\T^d;\R^d)$. We have
\begin{multline*}
g_\phi(v,w)
=\langle\bar\L(v\circ\phi^{-1}),w\circ\phi^{-1}\rangle+\int_{\T^d} B(v\circ\phi^{-1})\cdot B(w\circ\phi^{-1})\,\d x\\
=\langle\bar\L(v\circ\phi^{-1}),w\circ\phi^{-1}\rangle+\int_{\T^d} B(v\circ\phi^{-1})\circ\phi\cdot B(w\circ\phi^{-1})\circ\phi\,\det D\phi\,\d x.
\end{multline*}
Now $(w,\phi)\mapsto w\circ\phi^{-1}$ is $l$ times differentiable as a map from $H^m(\T^d;\R^d)\times\DiffSobIdLift^m$ to $H^{m-l}(\T^d;\R^d)$.
Indeed, it is linear in $w$ so that the map and all its derivatives with respect to $\phi$ are $C^\infty$-smooth in $w$ as long as they are continuous in $w$.
Thus it remains to check differentiability with respect to $\phi$.

Using the formulas
\begin{equation*}
\partial_\phi\phi^{-1}(u)=-D\phi^{-1}\,u\circ\phi^{-1},
\qquad
\partial_\phi(v\circ\phi^{-1})(u)=-Dv\circ\phi^{-1}\,D\phi^{-1}\,u\circ\phi^{-1}
\end{equation*}
we first show inductively that $\partial_\phi^l(v\circ\phi^{-1})(u_1,\ldots,u_l)$ is a sum of products of (potentially multiple factors) $D\phi^{-1}$
with (potentially zero or higher order) derivatives of $v$ and $u_1,\ldots,u_l$ (all composed with $\phi^{-1}$) and (potentially multiple, first or higher order) derivatives of $D\phi^{-1}$.
Moreover, these derivatives amount to an overall order of $l$ in each product, at most $l-1$ of which can act on $D\phi^{-1}$.
Indeed, this holds for $l=1$, and if it holds for $l$,
then applying another time $\partial_\phi$ in some direction $u_{l+1}$ leads by the product rule again to a sum of terms,
each one resulting from differentiating one factor of the previous products with respect to $\phi$.
If this factor is $D^jw\circ\phi^{-1}$ for $w$ being one of $v,u_1,\ldots,u_l$, then $$\partial_\phi[D^jw\circ\phi^{-1}](u_{l+1})=-D^{j+1}w\circ\phi^{-1}\,D\phi^{-1}\,u_{l+1}\circ\phi^{-1}$$
so that the total order of derivatives of $v,u_1,\ldots,u_{l+1},D\phi^{-1}$ indeed was increased by one (from $l$ to $l+1$).
If the factor is $D^j(D\phi^{-1})$, then $$\partial_\phi[D^j(D\phi^{-1})](u_{l+1})=-D^{j+1}(D\phi^{-1}\,u_{l+1}\circ\phi^{-1}),$$
which by the product rule and the same arguments as above increases the total order of derivatives of $v,u_1,\ldots,u_{l+1},D\phi^{-1}$ again by one
(and the highest occurring order of differentiating $D\phi^{-1}$ at most by one), as desired.

In summary, for $l\geq1$ we have that $\partial_\phi^l(v\circ\phi^{-1})(u_1,\ldots,u_l)$ is a sum of products
with one factor in each of $H^{m-n_1},\ldots,H^{m-n_k}$ (the differentiated $v$ and $u_j$), one factor in each of $H^{m-1-n_{k+1}},\ldots,H^{m-1-n_K}$ (the differentiated $D\phi^{-1}$)
as well as (potentially multiple) factors in $H^m$ (undifferentiated $u_j$) and $H^{m-1}$ (the factors $D\phi^{-1}$),
where $n_1,\ldots,n_K\geq1$ with $n_1+\ldots+n_K=l$ and $n_{k+1},\ldots,n_K\leq l-1$.
Now \cref{thm:BeHo21} implies that as soon as one function has Sobolev regularity larger than $\frac d2$, the product of two functions has the minimum Sobolev regularity of the two.
So assume first $m-n_1,\ldots,m-n_k,m-1-n_{k+1},\ldots,m-1-n_K>\frac d2$, then the product has Sobolev regularity $\min\{m-n_1,\ldots,m-n_k,m-1-n_{k+1},\ldots,m-1-n_K\}\geq m-l$.
Otherwise, if at least one factor has Sobolev regularity no larger than $\frac d2$, all factors with Sobolev regularity larger than $\frac d2$ can be ignored for the regularity of the product.
So assume without loss of generality that $m-n_1,\ldots,m-n_k,m-1-n_{k+1},\ldots,m-1-n_K\leq\frac d2$ (if this only holds for a subset, then reindex and redefine $k,K$ correspondingly).
Furthermore, we may assume $k=1$ without loss of generality (since for fixed $n_1,\ldots,n_K$ this represents the case of overall lowest Sobolev regularity).
Then, by iteratively applying \cref{thm:BeHo21} (starting with the most regular function and multiplying one after another the functions in the order of decreasing Sobolev regularity)
we obtain that the product has Sobolev regularity $m-n_1-\ldots-n_K+(K-1)(m-1-\frac d2)-\epsilon$ for any $\epsilon>0$ (note that \cref{thm:BeHo21} also holds for fractional exponents).
Due to $m>1+\frac d2$ and $n_1+\ldots+n_K\leq l$ this implies that the product has Sobolev regularity $m-l$.
Hence in total, $\partial_\phi^l(v\circ\phi^{-1})(u_1,\ldots,u_l)\in H^{m-l}$, where the dependence on $\phi,v,u_1,\ldots,u_l\in H^m$ is continuous;
consequently,
the $l$th derivative of $v\circ\phi^{-1}$ in $\phi$ is a bounded $l$-linear map from $H^m(\T^d;\R^d)^l$ into $H^{m-l}(\T^d;\R^d)$, which depends continuously on $\phi\in\DiffSobIdLift^m$.

Since the dual pairing is smooth (even bilinear), the map $(v,w,\phi)\mapsto\langle\bar\L(v\circ\phi^{-1}),w\circ\phi^{-1}\rangle$ is $l$ times differentiable
from $H^m(\T^d;\R^d)\times H^m(\T^d;\R^d)\times\DiffSobIdLift^m$ into the reals.

Further, the map $\phi\mapsto\det D\phi$ is smooth from $H^m(\T^d;\R^d)$ into $C^0(\T^d)$ as the composition of a $d$-degree polynomial with a linear operator.
Now, abbreviate $\bar Dw=D(w\circ\phi^{-1})\circ\phi$ (we drop the dependence on $\phi$ in the notation), then $D^k(w\circ\phi^{-1})\circ\phi=\bar D^kw$.
Since
\begin{equation*}
\bar Dw
=[Dw\circ\phi^{-1}\,D\phi^{-1}]\circ\phi
=Dw(D\phi)^{-1},
\end{equation*}
the map $(w,\phi)\mapsto\bar Dw$ is infinitely smooth from $H^s(\T^d;\R^d)\times\DiffSobIdLift^m$ into $H^{s-1}(\T^d)^{d\times d}$
(even though $\phi\mapsto\phi^{-1}$ is not smooth in $\DiffSobId^m$,
since matrix inversion and pointwise multiplication is).
By iteration, $(w,\phi)\mapsto\bar D^mw$ is smooth from $H^m(\T^d;\R^d)\times\DiffSobIdLift^m$ into $L^2(\T^d)^{d\times\cdots\times d}$.
Finally, $\phi\mapsto a_i\circ\phi$ is $l$-times differentiable from $\DiffSobIdLift^m$ into tensor fields of $L^\infty$-regularity
so that $(u,\phi)\mapsto B(u\circ\phi^{-1})\circ\phi$ is $l$ times differentiable from $H^m(\T^d;\R^d)\times\DiffSobIdLift^m$ into $L^2(\T^d)$-tensors.
Since pointwise product and integration are smooth as long as they are bounded, $(u,v,\phi)\mapsto g_\phi(u,v)$ is $l$ times differentiable.
\end{proof}

For our convergence result of the spectral space discretization, we will however require $\L$ to be a Fourier multiplier, so in that case the $a_i$ need to be constant in the previous result.

For reference, we also recall the result from \cite{GuRaRuWi23} on the weak PDE solved by geodesics
(there formulated for a domain $D\subset\R^d$), which is based on a similar decomposition of $\L$.
\begin{lemma}[{Weak geodesic PDE, \cite[Thm.\,1]{GuRaRuWi23}}]\label{thm:GuRaRuWi23}
Let $m>1+\frac d2$ and the Riemannian metric \eqref{eqn:metric} on $\DiffSob^m$ be induced by $\L=B^*B+\bar\L$ with $\bar\L:H^m(\T^d;\R^d)\to H^{1-m}(\T^d;\R^d)$ bounded linear
and $B=\sum_{k=0}^{m}a_kD^k$ a differential operator of order $m$ with coefficient tensors $a_i$ of $W^{0,\infty}$-regularity.
Then geodesics in $\DiffSob^m$ with velocity $v$ and momentum $\rho$ satisfy the weak PDE
\begin{equation*}
0=\int_0^1\!\!-2\langle\bar\L v_t, Dv_t\eta_t\rangle + 2\langle\rho_t,\dot\eta_t+D\eta_tv_t\rangle+\int_{\T^d}|Bv_t|^2\div\eta_t+2Bv_t\cdot{\mathcal U}(v_t,\eta_t)\,\d x\,\d t
\end{equation*}
for all smooth $\eta:[0,1]\times\T^d\to\R^d$, where $${\mathcal U}(v_t,\eta_t)=\frac\d{\d\epsilon}\left[B(v_t\circ(\id+\epsilon\eta_t)^{-1})\circ(\id+\epsilon\eta_t)\right]_{\epsilon=0}\in H^0((0,1)\times\T^d;\R^d).$$
Its strong form is \eqref{eqn:EPDiff}.
\end{lemma}
In fact, the proof does not require the specific form of $B$, but only uses differentiability of $H^m\times\DiffSob^m\ni(v,\phi)\mapsto[B(v\circ\phi^{-1})]\circ\phi\in L^2$.

We finally prove a new regularity result showing that along a geodesic, surplus Sobolev regularity of the initial velocity $v_0$ is preserved.
Moreover, in that case the formal geodesic equation \eqref{eqn:EPDiff} (or equivalently its integrated version) holds rigorously.
This will allow a convergent approximation via a spectral discretization.

\begin{theorem}[Preservation of higher Sobolev regularity]\label{thm:regularityPreservationLDDMM}
Let $m>\frac d2+1$, $1\leq k\leq m$, and the metric \eqref{eqn:metric} satisfy the conditions of \cref{thm:smoothMetric} for $l=2$ (thus be twice differentiable).
Let $\phi_t$ be a geodesic in $\DiffSobId^m$ with $\phi_0=\id$, and let $v_t\in H^m(\T^d;\R^d)$ be its velocity and $\rho_t=\L v_t\in H^{-m}(\T^d;\R^d)$ its momentum.
If $\rho_0\in H^{-m+k}(\T^d;\R^d)$ or equivalently $v_0\in H^{m+k}(\T^d;\R^d)$, then
\begin{equation*}
\rho_t=\adjointMap_{\phi_t^{-1}}^*\rho_0\in H^{-m+k}(\T^d;\R^d)
\quad\text{and }v_t\in H^{m+k}(\T^d;\R^d)
\qquad\text{for all }t>0.
\end{equation*}
Moreover, $t\mapsto(\rho_t,v_t)$ is continuous into $H^{-m+k}(\T^d;\R^d)\times H^{m+k}(\T^d;\R^d)$.
\end{theorem}

Note that for an $m+2$ times differentiable Riemannian metric (and thus $m$ times differentiable Riemannian exponential map)
this regularity preservation result already follows from \cite[main theorem and Sec.\,5.2]{Br17}.
In contrast to our direct computation, there the proof is based on an abstract argument trading in differentiability of the Riemannian exponential for preservation of Sobolev exponents, exploiting the right-invariance of the metric.

\begin{proof}
First note that by \cref{thm:normEstDeform}\eqref{enm:adjointMapEstimate} the adjoint map with respect to any $\psi\in\DiffSobId^m$ represents a bounded isomorphism
\begin{equation*}
\adjointMap_\psi:H^{m-k}(\T^d;\R^d)\to H^{m-k}(\T^d;\R^d).
\end{equation*}
Indeed, $\adjointMap_{\psi}^{-1}=\adjointMap_{\psi^{-1}}$ satisfies the same boundedness properties as $\adjointMap_\psi$ due to $\psi^{-1}\in\DiffSobId^m$.
Next simply define
\begin{equation*}
\rho_t=\adjointMap_{\phi_t^{-1}}^*(\rho_0)\in H^{k-m}(\T^d;\R^d),
\qquad
v_t=\Riesz\rho_t=\L^{-1}\rho_t\in H^{m+k}(\T^d;\R^d),
\end{equation*}
which by the previous is well-defined. For $w$ smooth we then obtain
\begin{multline*}
\langle\dot\rho_t,w\rangle
=\tfrac\d{\d t}\langle\rho_0,\adjointMap_{\phi_t^{-1}}w\rangle
=\langle\rho_0,(D\phi_t)^{\!-1}Dw\circ\phi_t\dot\phi_t-(D\phi_t)^{\!-1}\!D\dot\phi_t(D\phi_t)^{\!-1}w\circ\phi_t\rangle\\
=\langle\adjointMap_{\phi_t}^*\rho_t,\adjointMap_{\phi_t^{-1}}(Dwv_t-Dv_tw)\rangle
=\langle\rho_t,Dwv_t-Dv_tw\rangle,
\end{multline*}
which is still well-defined due to $v_t\in H^{m+k}(\T^d;\R^d)$.
However, this is exactly the weak form of the geodesic ODE \eqref{eqn:EPDiff} which by \cref{thm:GuRaRuWi23} rigorously characterizes the geodesics.
Since the geodesic ODE is uniquely solvable by the regularity of the metric, $\rho_t$ must be the corresponding momentum.
The continuous time dependence finally follows from \cref{thm:continuityAdjointAdjointMap} and the continuity of the geodesic in $\DiffSobId^m$.
\end{proof}

\begin{remark}[Still higher regularity]\label{rem:higherRegularity}
We formulated our result for $k\leq m$ since in that range the momentum $\rho_t$ stays a distribution.
However, the argument extends to $m<k\leq m+\frac d2$
(and the result is known to hold for $k>m+\frac d2$, see later).
Indeed, by \cref{thm:regularityPreservationLDDMM} we already know $\rho_t=\adjointMap_{\phi_t^{-1}}^*\rho_0$,
and it remains to show that $\rho_0\in H^{-m+k}(\T^d;\R^d)$ implies $\rho_t\in H^{-m+k}(\T^d;\R^d)$ for all $t$.
Now let $w:\T^d\to\R^d$ be smooth, then
\begin{align*}
\int_{\T^d}\rho_t\cdot w\,\d x
&=\int_{\T^d}\rho_0\cdot(D\phi_t)^{-1}w\circ\phi_t\,\d x
=\int_{\T^d}\rho_0\circ\phi_t^{-1}\cdot D\phi_t^{-1}w\,\det D\phi_t^{-1}\,\d x\\
&\leq\|w\|_{H^{m-k}}\|\det D\phi_t^{-1}\,(D\phi_t^{-1})^T\rho_0\circ\phi_t^{-1}\|_{H^{k-m}}\\
&\lesssim\|w\|_{H^{m-k}}\|\det D\phi_t^{-1}\|_{H^{m-1}}\|D\phi_t^{-1}\|_{H^{m-1}}\|\rho_0\circ\phi_t^{-1}\|_{H^{k-m}},
\end{align*}
where the last line follows from \cref{thm:BeHo21}, using $m-1>\frac d2\geq k-m$.
Since the determinant is a degree $d$ polynomial and $H^{m-1}$ is a Banach algebra (cf.\ \cref{thm:BeHo21}),
we have $\|\det D\phi^{-1}\|_{H^{m-1}}\lesssim\|D\phi^{-1}\|_{H^{m-1}}^d\lesssim(1+\|\phi^{-1}-\id\|_{H^m})^d$.
Furthermore, $\|\rho_0\circ\phi_t^{-1}\|_{H^{k-m}}\lesssim c(\phi_t^{-1})\|\rho_0\|_{H^{k-m}}$ by \eqref{eqn:regularityComposition}
so that overall
\begin{equation*}
\int_{\T^d}\rho_t\cdot w\,\d x
\leq C(\phi_t^{-1})\|\rho_0\|_{H^{k-m}}\|w\|_{H^{m-k}}
\end{equation*}
with a finite constant $C(\phi_t^{-1})$ depending only on $\phi_t^{-1}\in\DiffSobId^m$.
Thus, $\rho_t\in H^{k-m}(\T^d;\R^d)$ as desired.

For $k>m+\frac d2$ it was already shown in \cite[Thm.\,4.1]{MiPr10} by direct estimates
that the solution to the geodesic ODE \eqref{eqn:EPDiff} stays in $H^{k-m}(\T^d;\R^d)$ if $\rho_0$ is.
\end{remark}

\begin{remark}[Weak metrics]\label{rem:weakMetrics}
\Cref{thm:regularityPreservationLDDMM,rem:higherRegularity} solve a conjecture of \cite{MiPr10}:
The article considers right-invariant Sobolev metrics \eqref{eqn:metric} of order $m$ on $\DiffSobId^{m+k}$ with $k>0$.
Those are known as weak Riemannian metrics since they generate a topology on the tangent spaces $T_\phi\DiffSobId^{m+k}=H^{m+k}(\T^d;\R^d)$
with respect to which the tangent space is not complete.
As a consequence, the resulting Riemannian manifold is not metrically complete,
but nevertheless the geodesic equation and the Riemannian exponential map may still be well-defined.
Indeed, \cite[Thm.\,4.1]{MiPr10} shows that for a smooth metric and $k>m+\frac d2$ the Riemannian exponential $\exp_\phi^{m,k}:T_\phi\DiffSobId^{m+k}\to\DiffSobId^{m+k}$ is a local diffeomorphism\footnote{In fact, the proof of the diffeomorphism property in \cite[Thm.\,4.1]{MiPr10} does not require smoothness of the metric,
but just solvability of the geodesic ODE and differentiable dependence of its solution on the initial data, which holds as soon as the metric is three times differentiable.},
and this is conjectured to be true for any $k\geq0$.
\Cref{thm:regularityPreservationLDDMM,rem:higherRegularity} confirm this conjecture:
Indeed, they show that the restriction of $\exp_\phi^{m,0}$ from $T_\phi\DiffSobId^m$ to $T_\phi\DiffSobId^{m+k}$ has range in $\DiffSobId^{m+k}$
(recall that by \cite{BrVi17} the flow of a $H^{m+k}$-regular velocity lies in $\DiffSobId^{m+k}$),
thus $\exp_\phi^{m,k}$ is nothing else but this restriction and therefore well-defined.
If the metric \eqref{eqn:metric} is three times differentiable, it is even a local diffeomorphism,
which follows from the Inverse Function Theorem (noting that differentiability of the metric on $\DiffSobId^m$ implies differentiability on $\DiffSobId^{m+k}$).
\end{remark}

\begin{remark}[Non-integer exponents]
The results of this section, in particular \cref{thm:regularityPreservationLDDMM}, also hold for non-integer $m$ and $k$ since all employed estimates essentially go back to boundedness of composition with and inversion of Sobolev diffeomorphisms,
which can e.g.\ be found in \cite[Lemma\,B.5-B.6]{InKaTo13} (on $\R^d$, but the extension to bounded domains or manifolds is discussed there, too).
Of course, \cref{thm:smoothMetric} could no longer be applied in \cref{thm:regularityPreservationLDDMM} to guarantee twice differentiability of the metric,
but would have to be replaced by a version in which $B$ is a bounded operator from the \emph{fractional} Sobolev space $H^m$ to $L^2$
such that $(w,\phi)\mapsto[B(w\circ\phi^{-1}]\circ\phi$ is $l$ times differentiable
(which is the only property used in the proof of \cref{thm:smoothMetric}; \cite[Thm.\,1]{GuRaRuWi23}, used in our proof of \cref{thm:regularityPreservationLDDMM}, even just uses differentiability in $\phi$).
By \cite[Thm.\,4.16 \& Thm.\,5.1]{BaBrCiEsKo20} this is for instance satisfied for any $B\in\mathrm{OPS}_{1,0}^{m}$.

Also the precise control of the constants in \cref{thm:Regulcomposition} and the following carry over:
In the proof of \cref{thm:Regulcomposition} one would begin the induction with
\begin{align*}
\|f\circ\phi\|_{H^\theta}
&=\int_{\T^d}\int_{\T^d}\frac{|f(\phi(x))-f(\phi(y))|^2}{|x-y|^{d+2\theta}}\,\d x\,\d y\\
&=\int_{\T^d}\int_{\T^d}\frac{|f(x)-f(y)|^2}{|\phi^{-1}(x)-\phi^{-1}(y)|^{d+2\theta}}\det D(\phi^{-1})(x)\det D(\phi^{-1})(y)\,\d x\,\d y\\
&\leq[(1+C\|\phi-\id\|_{H^m})\|f\|_{H^\theta}]^2
\end{align*}
for $\theta\in(0,1)$ and some $C>0$, exploiting $\det D(\phi^{-1})\leq1+C\|\phi-\id\|_{H^m}$ (cf.\ the proof of \cref{thm:Regulcomposition}) and $|\phi^{-1}(x)-\phi^{-1}(y)|\geq|x-y|/(1+L)$
with $L\lesssim\|\phi-\id\|_{H^m}$ the Lipschitz constant of $\phi-\id$ (as can be readily seen after the change of variables $X=\phi^{-1}(x)$, $Y=\phi^{-1}(y)$).
\end{remark}

\section{Spectral space discretization and its properties}\label{sec:discretization}
In this section we introduce the spectral space discretization of the geodesic equation \eqref{eqn:EPDiff} (a simplified variant of the one from \cite{ZhFl19}) and analyse its structure.
To this end we employ the semi-discrete Fourier transform or Fourier series transform
\begin{equation*}
\hat f(\xi)=\int_{\T^d} f\exp(-2\pi i\xi\cdot x)\,\d x\qquad\text{for }\xi\in\Z^d
\end{equation*}
of functions $f:\T^d\to\R$ and of distributions on $\T^d$ (for vector- or matrix-valued functions the Fourier transform is just applied componentwise).
Fixing a maximum frequency $R$, we introduce the Fourier series truncated at frequency $R$ by
\begin{equation*}
\truncFourier{f}{R}(\xi)=\begin{cases}\hat f(\xi)&\text{if }|\xi|_\infty\leq R,\\0&\text{else}\end{cases}
\end{equation*}
and denote the corresponding truncation operator as $\trunc$, defined via
\begin{equation*}
\widehat{\trunc f}=\truncFourier fR.
\end{equation*}
The consistency order of $\trunc$ obviously depends on the Sobolev regularity.

\begin{lemma}[Consistency order of truncated Fourier series]\label{thm:consistencyTruncatedFourier}
For $k\geq l\geq0$ we have
\begin{equation*}
\|\trunc f-f\|_{H^l}\lesssim R^{l-k}\|\trunc f-f\|_{H^k}\leq R^{l-k}\|f\|_{H^k}.
\end{equation*}
\end{lemma}
\begin{proof}
Without loss of generality we consider the norm $\|f\|_{H^k}^2=\langle f,(1+(\frac{-\Delta}{4\pi^2})^k)f\rangle$ on $H^k(\T^d)$ and analogously on $H^l(\T^d)$.
Let $f\in H^k(\T^d)$, then
$\|\trunc f-f\|_{H^l}^2
=\sum_{|\xi|_\infty>R}|\hat f(\xi)|^2(1+|\xi|^{2l})
\leq\frac{1+R^{2l}}{1+R^{2k}}\sum_{|\xi|_\infty>R}|\hat f(\xi)|^2(1+|\xi|^{2k})
\lesssim R^{2(l-k)}\|\trunc f-f\|_{H^k}^2$.
\end{proof}

Analogously one obtains
\begin{equation}\label{eqn:normEstimateTruncation}
\|\trunc f\|_{H^k}\lesssim R^{k-l}\|\trunc f\|_{H^l}
\qquad\text{for }k\geq l\geq0.
\end{equation}
In case $f\in H^k(\T^d)$ has no additional regularity one only knows $\|\trunc f-f\|_{H^k}\to0$ as $R\to\infty$, but one cannot give a rate.
At least one obtains uniform convergence for a continuous family of functions
(below, recall that $\Id$ denotes the identity operator).

\begin{lemma}[Uniform Fourier decay along a continuous path]\label{thm:FourierDecay}
Let $k\in\Z$ and $w:[0,1]\to H^k(\T^d)$ be continuous, then $\|(\Id-\trunc)w_t\|_{H^k}\to0$ as $R\to\infty$, uniformly in $t\in[0,1]$.
\end{lemma}
\begin{proof}
For a proof by contradiction assume there exists some $C>0$ as well as times $t_R\in[0,1]$, $R\in\N$, such that $\|(\Id-\trunc)w_{t_R}\|_{H^k}>C$ for all $R$.
By compactness of $[0,1]$ the sequence $t_R$ contains a converging subsequence with limit $t\in[0,1]$.
Since $\|(\Id-\trunc[r])w_{t_R}\|_{H^k}>C$ for all $r\leq R$ we may assume without loss of generality that the whole sequence converges.
Let $r\in\N$ be such that $\|(\Id-\trunc[r])w_t\|_{H^k}<C/2$,
then for $R>r$ we have
\begin{equation*}
C
<\|(\Id-\trunc)w_{t_R}\|_{H^k}
\leq\|(\Id-\trunc[r])w_{t_R}\|_{H^k}
\mathop\to_{R\to\infty}\|(\Id-\trunc[r])w_t\|_{H^k}
<C/2,
\end{equation*}
where the limit follows from the continuity of $(\Id-\trunc[r])$ and $t\mapsto w_t$,
yielding the desired contradiction.
\end{proof}

The formal geodesic equation \eqref{eqn:EPDiff} in Fourier space reads
\begin{equation*}
\dot{\hat\rho}_t=-((\hat\rho_t\otimes\hat D)*\hat v_t+\hat\rho_t*(\hat D\cdot\hat v_t)+(\hat D\otimes\hat v_t)*\hat\rho_t)
\qquad\text{with }
\hat v_t=\hat\Riesz\hat\rho_t
\end{equation*}
or equivalently
\begin{equation}\label{eqn:geodesicEquationLDDMMFourier}
\dot{\hat v}_t
=-\hat\Riesz((\hat\rho_t\otimes\hat D)*\hat v_t+\hat\rho_t*(\hat D\cdot\hat v_t)+(\hat D\otimes\hat v_t)*\hat\rho_t)
\qquad\text{with }\quad
\hat\rho_t=\hat\L\hat v_t,
\end{equation}
where $*$ denotes discrete convolution (to be interpreted in the obvious way if a matrix field is convolved with a vector field),
$\hat D(\xi)=2\pi i\xi$ is the Fourier multiplier of differentiation,
and $\hat\Riesz$ and $\hat\L=\hat\Riesz^{-1}$ are just the Riesz operator and its inverse expressed as operators on Fourier space.
In the following we will assume the Riesz operator and its inverse to be Fourier multipliers
so that the relation $v_t=\Riesz\rho_t$ between velocity and momentum remains after Fourier truncation, $\trunc v_t=\Riesz \trunc\rho_t$.
Therefore, $\hat\Riesz$ and $\hat\L$ may simply be interpreted as functions on $\Z^d$ in the above.

The spatial discretization of \eqref{eqn:geodesicEquationLDDMMFourier} is merely to replace the Fourier transforms of $v_0$ and $\Riesz$ by their truncated versions, thus
\begin{equation}\label{eqn:LDDMMdiscreteFourier}
\dot{\hat V}_t
=-\truncFourier\Riesz R((\hat P_t\otimes\hat D)*\hat V_t+\hat P_t*(\hat D\cdot\hat V_t)+(\hat D\otimes\hat V_t)*\hat P_t)
\quad\text{ for }\quad
\hat P_t=\hat\L\hat V_t
\end{equation}
and $\hat V_0=\truncFourier vR_0$, where we use capital letters $V$ and $P$ for the numerical approximation of $v$ and $\rho$.
By construction, $\hat V_t$ and $\hat P_t$ have support on $Z_{d,R}=\{\xi\in\Z^d\,|\,|\xi|_\infty\leq R\}$.
This spectral discretization is particularly popular due to its computational efficiency.

\begin{proposition}[Computational effort]\label{thm:effortLDDMM}
The right-hand side of \eqref{eqn:LDDMMdiscreteFourier} can be evaluated in $O(R^d\log R)$ time.
\end{proposition}
\begin{proof}
Since $\hat P_t$ and $\hat V_t$ are supported on $Z_{d,R}$,
all pointwise multiplications of $\hat P_t$, $\hat V_t$, and $\truncFourier\Riesz R$ require $O(R^d)$ floating point operations, likewise the multiplications with $\hat\L$ and $\hat D$.
It remains to compute convolutions of the form $\hat f*\hat g$ with $\hat f$ and $\hat g$ supported on $Z_{d,R}$; to this end we extend $\hat f$ and $\hat g$ to $Z_{d,2R}$ with zeros
and simply perform a circular convolution with subsequent truncation to $Z_{d,R}$.
Using the Fast Fourier Transform and the Discrete Convolution Theorem this requires $O(R^d\log R)$ operations.
\end{proof}

The (space-discrete) ODE \eqref{eqn:LDDMMdiscreteFourier} is equivalently expressed as
\begin{equation}\label{eqn:LDDMMdiscrete}
\dot P_t=-\trunc{}[\div(P_t\otimes V_t)+(DV_t)^TP_t]
=-\trunc\,\adjointDer_{V_t}^*P_t
\quad\text{ for }\quad
V_t=\Riesz P_t.
\end{equation}
As already observed in \cite{ZhFl19}, 
this discrete approximation may be viewed as geodesic ODE in a finite-dimensional ``approximate'' Riemannian Lie group.
Here and in the following, ``approximate'' shall only refer to the fact that the associated bracket has a nonzero, but bounded Jacobiator,
i.e.\ the Jacobi identity is only satisfied up to a bounded error.%
\footnote{In the context of deformation theory for Lie groups and Lie algebras it seems
that closeness to a Lie algebra can sometimes be quantified by the smallness of the Jacobiator (the deviation from the Jacobi identity),
however, in our setting the Jacobiator is not small, but scales like $R^2$.}

\begin{proposition}[Approximate Riemannian Lie group geodesics]
The discrete approximation \eqref{eqn:LDDMMdiscrete} is equivalent to the Euler--Poincar\'e equation
\begin{equation}\label{eqn:LDDMMapproximateLie}
\dot P_t=-\widetilde{\adjointDer}_{V_t}^*P_t,\quad V_t=\Riesz P_t,
\end{equation}
on $\truncSpace=\trunc H^m(\T^d;\R^d)$ with inner product induced by the canonical embedding $\truncSpace\hookrightarrow H^m(\T^d;\R^d)$
and approximate Lie bracket $\widetilde{\adjointDer}_{V}W=-[V,W]_R=-\trunc([V,W])$.
\end{proposition}
\begin{proof}
To compare the solutions of \eqref{eqn:LDDMMapproximateLie} and \eqref{eqn:LDDMMdiscrete} it obviously suffices to work with functions from $\truncSpace$,
since the solution of \eqref{eqn:LDDMMdiscrete} lies in this space.
Thus, let $W\in\truncSpace$ and let $P_t,V_t\in\truncSpace$ solve \eqref{eqn:LDDMMapproximateLie}, then
\begin{equation*}
\langle\dot P_t,W\rangle
=-\langle P_t,\widetilde{\adjointDer}_{V_t}W\rangle
=-\langle P_t,\adjointDer_{V_t}W\rangle
=\langle-\adjointDer_{V_t}^*P_t,W\rangle.
\end{equation*}
Since this holds for arbitrary $W\in\truncSpace$ we have $\dot P_t=\trunc\dot P_t=-\trunc\,\adjointDer_{V_t}^*P_t$, which coincides with \eqref{eqn:LDDMMdiscrete}.
\end{proof}

\begin{remark}[Obstruction to finite-dimensional Lie group approximations]
The Lie bracket $[v,w]$ satisfies the three characterizing properties of a Lie bracket:
bilinearity, antisymmetry, and the Jacobi identity.
However, as discussed before, it is not closed in $H^m(\T^d;\R^d)$, i.e.\ $v,w\in H^m(\T^d;\R^d)$ does not imply $[v,w]\in H^m(\T^d;\R^d)$.
In contrast, the approximate Lie bracket $[V,W]_R$ is closed in $\truncSpace$,
but it is not a legitimate Lie bracket as it does not satisfy the Jacobi identity.
This is due to a principal obstruction observed by Omori in \cite{Om78}:
There exists no Lie algebra $\truncSpace$ with finite Fourier support and correct Lie bracket.
Let us quickly adapt the argument from \cite{Om78} to our setting, for simplicity restricting to dimension $d=1$:
Since the Lie bracket is closed, as soon as the Lie algebra $\truncSpace$ contains trigonometric functions of two different nonzero frequencies,
then by repeatedly applying the Lie bracket $[\cdot,\cdot]$ one obtains trigonometric functions of arbitrarily high frequencies, which therefore must belong to $\truncSpace$.
It does not help to drop the condition on $\truncSpace$ of finite Fourier support, either:
There cannot exist any Hilbert Lie algebra containing the constant function and trigonometric functions of arbitrarily high frequencies.
Indeed, due to $[[1,\sin(2\pi\xi x)],1]=[2\pi\xi\cos(2\pi\xi x),1]=(2\pi\xi)^2\sin(2\pi\xi x)$ the Lie bracket cannot be bounded in any norm on $\truncSpace$.
The only exception are the spaces $\truncSpace=\Span(1,\sin(2\pi\xi x),\cos(2\pi\xi x))$ with exactly one frequency $\xi\in\Z$;
on those, the standard Lie bracket is closed so that $\truncSpace$ is a true Lie algebra.
\end{remark}

By construction, the discrete approximation \eqref{eqn:LDDMMdiscrete} is structure-preserving in that -- just like for the original geodesic equation -- the norm of the velocity is conserved.

\begin{lemma}[Constant velocity geodesics]\label{thm:normPreservation}
The solution $V_t$ to \eqref{eqn:LDDMMdiscrete} conserves $\NHm{V_t}$ over time.
\end{lemma}
\begin{proof}
In fact, this holds for any equation of the form \eqref{eqn:LDDMMapproximateLie} with antisymmetric Lie bracket approximation,
i.e.\ with $\widetilde{\adjointDer}_{V}W=-\widetilde{\adjointDer}_{W}V$.
Indeed, let $P_t,V_t$ denote the solution of \eqref{eqn:LDDMMapproximateLie}, then
\begin{equation*}
\frac\d{\d t}\frac{\NHm{V_t}^2}2
=\langle V_t,\dot P_t\rangle
=-\langle\widetilde{\adjointDer}_{V_t}V_t,P_t\rangle
=0.
\qedhere
\end{equation*}
\end{proof}

\section{Convergence of spectral discretization}\label{sec:convergence}
In this section we prove convergence of the spectral discretization as well as convergence rates depending on the geodesic regularity.
We begin with a regularity estimate for quadratic terms that occur in the geodesic equation.

\begin{proposition}[Quadratic term regularity]\label{thm:productsInB}
Let $m>\frac d2+1$ and let $\L$ satisfy the conditions from \cref{thm:smoothMetric} for $l=1$.
Then 
\begin{equation*}
\|w^TD\L w\|_{H^{-m-1}}\lesssim\|w\|_{H^m}^2
\qquad\text{for all }w\in H^m(\T^d;\R^d).
\end{equation*}
\end{proposition}
\begin{proof}
For simplicity we assume $B=a_mD^m$ (additional lower derivatives are treated analogously).
It is straightforward to see $\partial_{x_i}B^*h=B^*\partial_{x_i}h+(D^m)^*[(\partial_{x_i}a_m^*)h]$ for any function $h$,
where $a_m^*$ is the adjoint of the tensor $a_m$.
Thus, for $g\in H^{m+1}(\T^d)$ we have
\begin{equation*}
\int_{\T^d}\!\!gw^T\partial_{x_i\!}\L w\,\d x
\!=\!\int_{\T^d}\!\!gw\cdot\partial_{x_i\!}(\bar\L w)\,\d x
+\!\int_{\T^d}\!\! \partial_{x_i\!}Bw\cdot B(gw)+D^m(gw)(\partial_{x_i\!}a_m^*)\cdot Bw\,\d x.
\end{equation*}
The first term is clearly bounded in absolute value by a constant times $\|gw\|_{H^m}\|\partial_{x_i}\bar\L w\|_{H^{-m}}\lesssim\|w\|_{H^m}^2\|g\|_{H^m}$ (recall that $H^m$ is a Banach algebra).
Furthermore, by the product rule $B(gw)$ equals $(Bw)g$ plus a sum $S$ of products of derivatives of $w$ and $g$, each summand with $m$ derivatives in total, at least one of them on $g$.
Thus,
\begin{align*}
\int_{\T^d} \partial_{x_i}Bw\cdot B(gw)\,\d x
&=\int_{\T^d} g\,\partial_{x_i}\tfrac{|Bw|^2}2\,\d x
+\int_{\T^d}S\cdot\partial_{x_i}Bw\,\d x\\
&=-\int_{\T^d}\tfrac{|Bw|^2}2\partial_{x_i}g\,\d x
+\int_{\T^d}S\cdot\partial_{x_i}Bw\,\d x.
\end{align*}
The first integral is bounded by $\|Bw\|_{H^0}^2\|\partial_{x_i}g\|_{W^{0,\infty}}\lesssim\|w\|_{H^m}^2\|g\|_{H^{m+1}}$,
the second by $\|\partial_{x_i}Bw\|_{H^{-1}}\|S\|_{H^1}\lesssim\|w\|_{H^m}^2\|g\|_{H^{m+1}}$ (using \cref{thm:BeHo21}).
Similarly,
\begin{multline*}
\left|\int_{\T^d}D^m(gw)(\partial_{x_i}a_m^*)\cdot Bw\,\d x\right|
\leq\|Bw\|_{H^0}\|\partial_{x_i}a_m\|_{W^{0,\infty}}\|D^m(gw)\|_{H^0}\\
\lesssim\|w\|_{H^m}\|gw\|_{H^m}
\lesssim\|w\|_{H^m}^2\|g\|_{H^m}.
\end{multline*}
Hence, $\left|\int_{\T^d} gw^T\partial_{x_i}\L w\,\d x\right|\lesssim\|w\|_{H^m}^2\|g\|_{H^{m+1}}$ for all $g\in H^{m+1}(\T^d)$ and all coordinates $i$
so that the result follows from $w^TD\L w=\sum_{i=1}^dw^T\partial_{x_i}\L w$.
\end{proof}

A direct consequence is a regularity result for the right-hand side of the geodesic ODE \eqref{eqn:EPDiff}.

\begin{corollary}[$H^{m-1}$-regularity of velocity change]\label{thm:rhsRegularity}
Let $m\!>\!\frac d2\!+\!1$, $k\!\geq\!0$, let $\L$ satisfy the conditions from \cref{thm:smoothMetric} for $l=1$, and let $\rho$ satisfy \eqref{eqn:EPDiff}.
Then $\|\dot\rho_t\|_{H^{k-m-1}}\lesssim\|\rho_t\|_{H^{-m}}\|\rho_t\|_{H^{k-m}}$ or equivalently $\|\dot v_t\|_{H^{m-1+k}}\lesssim\|v_t\|_{H^{m}}\|v_t\|_{H^{m+k}}$.
\end{corollary}
\begin{proof}
For $w\in H^{m+1-k}(\T^d;\R^d)$ we have
\begin{multline*}
\langle\dot\rho_t,w\rangle
=\langle\rho_t,[v_t,w]\rangle
=\langle\rho_t,Dw\,v_t\rangle-\int_{\T^d}\rho_t^TDv_t\,w\,\d x\\
=\langle\rho_t,Dw\,v_t\rangle+\langle\rho_t,v_t\div w\rangle+\langle v_t^TD\rho_t,w\rangle,
\end{multline*}
of which all summands can be bounded by $\|\rho_t\|_{H^{k-m}}\|v_t\|_{H^m}\|w\|_{H^{m+1-k}}$
(using \cref{thm:BeHo21} and for $k=0$ \cref{thm:productsInB} in the last summand).
\end{proof}

The convergence of the spectral discretization now follows via a typical Gronwall type ODE argument.

\begin{theorem}[Convergence of geodesic approximation]\label{thm:convergence}
Let $m>\frac d2+1$ and $\L$ be a Fourier multiplier satisfying the conditions from \cref{thm:smoothMetric} for $l=2$ (so that geodesics are well-posed in $\DiffSobId^m$).
Let $v_t$ be the velocity of a geodesic in $\DiffSobId^m$ and $V_t$ the solution of \eqref{eqn:LDDMMapproximateLie} with $V_0=\trunc v_0$.
\begin{enumerate}
\item
If $v_0\in H^{m+1}(\T^d;\R^d)$, then $\|V_t-v_t\|_{H^m}\to_{R\to\infty}0$ uniformly in $t\in[0,1]$.
\item
If $v_0\in H^{m+k}(\T^d;\R^d)$ for $k>1$, then
\begin{equation*}
\|V_t-v_t\|_{H^m}\leq Ce^{C\|v_0\|_{H^{m+1}}e^{C\|v_0\|_{H^m}}}\|v_0\|_{H^{m+k}}R^{1-k}
\end{equation*}
for all $t\in[0,1]$ with a constant $C>0$ independent of $v_0$.
\end{enumerate}
\end{theorem}
\begin{proof}
Extend $\widetilde\adjointDer$ to $H^m(\T^d;\R^d)\times H^m(\T^d;\R^d)$ by $\widetilde\adjointDer=\trunc\adjointDer$, which is still antisymmetric.
Let $e_t=V_t-v_t$.
We have
\begin{equation*}
\dot\rho_t=-\adjointDer_{v_t}^*\rho_t,\qquad
\dot P_t=-\trunc\widetilde\adjointDer_{V_t}^*P_t.
\end{equation*}
We obtain
\begin{multline*}
\frac\d{\d t}\frac{\NHm{e_t}^2}2
=\langle e_t,\L\dot e_t\rangle
=\langle\trunc e_t,\trunc\L\dot e_t\rangle
+\langle(\Id-\trunc)e_t,(\Id-\trunc)\L\dot e_t\rangle\\
=\langle\trunc e_t,\L\dot e_t\rangle
+\langle(\Id-\trunc)v_t,(\Id-\trunc)\dot\rho_t\rangle,
\end{multline*}
where we exploited $(\Id-\trunc)e_t=-(\Id-\trunc)v_t$.
We estimate the first term as follows,
\begin{multline*}
\langle\trunc e_t,\L\dot e_t\rangle
=\langle\trunc e_t,\adjointDer_{v_t}^*\rho_t-\trunc\widetilde\adjointDer_{V_t}^*P_t\rangle\\
=\langle[V_t,\trunc e_t]_R,P_t\rangle-\langle[v_t,\trunc e_t],\rho_t\rangle
=\langle[V_t-v_t,\trunc e_t],P_t\rangle+\langle[v_t,\trunc e_t],P_t-\rho_t\rangle\\
=-\langle[(\Id-\trunc)v_t,\trunc e_t],P_t\rangle+\IPHm{[v_t,\trunc e_t]}{\trunc e_t}-\IPHm{[v_t,\trunc e_t]}{(\Id-\trunc)v_t},
\end{multline*}
exploiting $[\trunc e_t,\trunc e_t]=0$ and therefore $[e_t,\trunc e_t]=[(\Id-\trunc)e_t,\trunc e_t]=-[(\Id-\trunc)v_t,\trunc e_t]$.
Using that $H^m$ is a Banach algebra as well as \cref{thm:consistencyTruncatedFourier} and \eqref{eqn:normEstimateTruncation} we next estimate
\begin{align*}
\|[(\Id\!-\!\trunc)v_t,\trunc e_t]\|_{H^m}
&=\|D\trunc e_t\,(\Id\!-\!\trunc)v_t-D(\Id\!-\!\trunc)v_t\,\trunc e_t]\|_{H^m}\\
&\lesssim\|(\Id\!-\!\trunc)v_t\|_{H^{m}}\|\trunc e_t\|_{H^{m+1}}+\|(\Id\!-\!\trunc)v_t\|_{H^{m+1}}\|\trunc e_t\|_{H^m}\\
&\lesssim\tfrac{\|(\Id\!-\!\trunc)v_t\|_{H^{m+1}}}RR\|\trunc e_t\|_{H^{m}}+\|(\Id\!-\!\trunc)v_t\|_{H^{m+1}}\|\trunc e_t\|_{H^{m}}\\
&\leq2\|(\Id\!-\!\trunc)v_t\|_{H^{m+1}}\|\trunc e_t\|_{H^{m}}.
\end{align*}
Therefore, exploiting \cref{thm:normPreservation} we have
\begin{multline*}
-\langle[(\Id\!-\!\trunc)v_t,\trunc e_t],P_t\rangle
\leq\|[(\Id\!-\!\trunc)v_t,\trunc e_t]\|_{H^m}\|P_t\|_{H^{-m}}\\
\lesssim\|(\Id\!-\!\trunc)v_t\|_{H^{m+1}}\NHm{e_t}\NHm{V_t}
=\NHm{e_t}\NHm{V_0}\|(\Id\!-\!\trunc)v_t\|_{H^{m+1}}\\
\leq\NHm{e_t}\NHm{v_0}\|(\Id\!-\!\trunc)v_t\|_{H^{m+1}}.
\end{multline*}
Furthermore, an integration by parts yields
\begin{align*}
\IPHm{[q,u]}{w}
&=\int_{\T^d} (\L w)^T(Du\,q-Dq\,u)\,\d x\\
&=-\int_{\T^d} \L w\cdot u\,\div q+u^TD(\L w)q+(\L w)^TDq\,u\,\d x
\end{align*}
for any $u\in H^m(\T^d;\R^d)$ and $q,w\in H^{m+1}(\T^d;\R^d)$.
Therefore we have
\begin{multline*}
\IPHm{[v_t,\trunc e_t]}{\trunc e_t}\\
=-\!\int_{\T^d}\!\!\!\L\trunc e_t\!\cdot\!\trunc e_t\div v_t\!+\!\trunc e_t^TD(\L\trunc e_t)v_t\!+\!(\L\trunc e_t)^T\!Dv_t\trunc e_t\,\d x
\lesssim\NHm{e_t}^2\|v_t\|_{H^{m+1}}
\end{multline*}
by \cref{thm:productsInB} as well as
\begin{align*}
&-\IPHm{[v_t,\trunc e_t]}{(\Id-\trunc)v_t}\\
&=\int_{\T^d}\!\!\!\L(\Id\!-\!\trunc)v_t\!\cdot\! \trunc e_t\div v_t\!+\!\trunc e_t^TD(\L(\Id\!-\!\trunc)v_t)v_t\!+\!(\L(\Id\!-\!\trunc)v_t)^T\!Dv_t\trunc e_t\,\d x\\
&\lesssim\NHm{e_t}\|v_t\|_{H^{m+1}}\|(\Id-\trunc)v_t\|_{H^{m+1}}.
\end{align*}
Summarizing,
\begin{multline}\label{eqn:ODEestimate}
\frac\d{\d t}\frac{\NHm{e_t}^2}2
\lesssim\NHm{e_t}\NHm{v_0}\|(\Id\!-\!\trunc)v_t\|_{H^{m+1}}
+\NHm{e_t}^2\|v_t\|_{H^{m+1}}\\
+\NHm{e_t}\|v_t\|_{H^{m+1}}\|(\Id-\trunc)v_t\|_{H^{m+1}}
+\|(\Id-\trunc)v_t\|_{H^{m+1}}\|(\Id-\trunc)\dot v_t\|_{H^{m-1}},
\end{multline}
where the last summand is an upper bound for $\langle(\Id-\trunc)v_t,(\Id-\trunc)\dot\rho_t\rangle$.
(Note that if we had substituted $\langle(\Id-\trunc)v_t,(\Id-\trunc)\dot\rho_t\rangle=\langle(\Id-\trunc)v_t,-\adjointDer_{v_t}^*\rho_t\rangle=\langle[v_t,(\Id-\trunc)v_t],\rho_t\rangle$,
the smallness of the term from two Fourier truncations rather than one would no longer be obvious.)
By assumption and \cref{thm:regularityPreservationLDDMM}, $\|v_t\|_{H^{m+1}}$ is uniformly bounded.
Furthermore, from \cref{thm:rhsRegularity} we see
\begin{equation*}
\|\dot v_t\|_{H^{m-1}}
\lesssim\NHm{v_t}^2
=\NHm{v_0}^2.
\end{equation*}
Since $t\mapsto v_t\in H^{m+1}(\T^d;\R^d)$ is continuous by \cref{thm:regularityPreservationLDDMM},
$\|(\Id-\trunc)v_t\|_{H^{m+1}}$ tends to zero uniformly in $t$ as $R\to\infty$ due to \cref{thm:FourierDecay}.
Since also $e_0\to0$ in $H^m(\T^d;\R^d)$, the Bihari--LaSalle inequality implies $e_t\to0$ as $R\to\infty$, uniformly in $t$, which proves the first claim.

For the second claim 
we collect the occurring powers of $R$; we start with
\begin{equation*}
\NHm{e_0}
=\NHm{(\Id-\trunc)v_0}
\lesssim R^{-k}\|v_0\|_{H^{m+k}}.
\end{equation*}
Furthermore, denoting by $\phi_t$ the flow of $v_t$, we have
\begin{multline*}
\|(\Id-\trunc)v_t\|_{H^{m+1}}
\lesssim R^{1-k}\|v_t\|_{H^{m+k}}
\lesssim R^{1-k}\|\rho_t\|_{H^{-m+k}}
=R^{1-k}\|\adjointMap_{\phi_t^{-1}}^*\rho_0\|_{H^{-m+k}}\\
\leq R^{1-k}\exp(C\|v\|_{L^1(I;H^m)})\|\rho_0\|_{H^{-m+k}}
=R^{1-k}\exp(C\NHm{v_0})\|v_0\|_{H^{m+k}},
\end{multline*}
where in the first step we used \cref{thm:consistencyTruncatedFourier},
in the second $\|\rho_t\|_{H^{-m+k}}=\|\L v_t\|_{H^{-m+k}}\lesssim\|v_t\|_{H^{m+k}}$,
in the third \cref{thm:regularityPreservationLDDMM},
in the fourth \cref{thm:normEstDeform}\eqref{enm:adjointMapEstimate} to estimate the operator norm of $\adjointMap_{\phi_t^{-1}}$,
and in the last step we used that the norm $\NHm{v_t}$ of the velocity is constant along a geodesic.
Similarly,
\begin{multline*}
\|(\Id-\trunc)\dot v_t\|_{H^{m-1}}
\lesssim R^{-k}\|\dot v_t\|_{H^{m+k-1}}
\lesssim R^{-k}\|\rho_t\|_{H^{-m+k}}\NHm{v_t}\\
\leq R^{-k}\exp(C\NHm{v_0})\|v_0\|_{H^{m+k}}\NHm{v_t}
=R^{-k}\exp(C\NHm{v_0})\|v_0\|_{H^{m+k}}\NHm{v_0}
\end{multline*}
using \cref{thm:rhsRegularity} in the second step. In summary, \eqref{eqn:ODEestimate} turns into
\begin{align*}
\frac\d{\d t}\frac{\NHm{e_t}^2}2
&\lesssim\|v_0\|_{H^{m+1}}\exp(2C\NHm{v_0})\cdot\\
&\quad\left(\NHm{e_t}^2+R^{1-k}\|v_0\|_{H^{m+k}}\sqrt{\NHm{e_t}^2}+R^{-2k+1}\tfrac{\|v_0\|_{H^{m}}\|v_0\|_{H^{m+k}}^2}{\|v_0\|_{H^{m+1}}}\right)\\
&\lesssim\|v_0\|_{H^{m+1}}\exp(2C\NHm{v_0})\left(\NHm{e_t}^2+R^{2-2k}\|v_0\|_{H^{m+k}}^2\right)
\end{align*}
via Young's inequality,
which together with the estimate for $\NHm{e_0}$ yields the desired estimate via the Gronwall inequality.
\end{proof}

For $v_0$ of generic $H^m$-regularity one does not get convergence of the scheme.
However, using the above rates one can devise a new scheme that converges even for $v_0\in H^m(\T^d;\R^d)$ and also gives a (logarithmically slow) convergence rate for $v_0\in H^{m+1}(\T^d;\R^d)$.

\begin{corollary}[Convergent scheme for generic regularity]\label{thm:convergenceGenericRegularity}
Let the conditions of \cref{thm:convergence} hold, and let $r\in o(\log R)$.
Given $v_0\in H^m(\T^d;\R^d)$, let $V^r\in\trunc H^{m}(\T^d;\R^d)$ denote the solution of \eqref{eqn:LDDMMapproximateLie} for initial condition $V^r_0=\trunc[r]v_0$.
Then $\|V^r_t-v_t\|_{H^m}\to0$ uniformly in $t\in[0,1]$ as $R\to\infty$.

Moreover, if $v_0\in H^{m+1}(\T^d;\R^d)$ and the metric \eqref{eqn:metric} is three times differentiable,
then $\|V^r_t-v_t\|_{H^m}\lesssim C(\|v_0\|_{H^{m+1}})r$ for all $t\in[0,1]$ for a constant depending only on $\|v_0\|_{H^{m+1}}$.
\end{corollary}
\begin{proof}
Denote by $v^r$ the geodesic with initial velocity $v^r_0=V^r_0=\trunc[r]v_0$.
We can estimate
\begin{equation*}
\|V^r_t-v_t\|_{H^m}
\leq\|V^r_t-v^r_t\|_{H^m}+\|v^r_t-v_t\|_{H^m}.
\end{equation*}
By \cref{thm:convergence}, for $k\geq2$ we have
\begin{align}
\|V^r_t-v^r_t\|_{H^m}
&\leq Ce^{C\|v_0^r\|_{H^{m+1}}e^{C\|v_0^r\|_{H^m}}}\|v_0^r\|_{H^{m+k}}R^{1-k}\nonumber\\
&\leq Ce^{Cr\|v_0\|_{H^{m}}e^{C\NHm{v_0}}}r^k\|v_0\|_{H^{m}}R^{1-k},\label{eqn:rR_estimate}
\end{align}
which for $r\in o(\log R)$ tends to zero as $R\to\infty$.
Likewise, $\|v^r_t-v_t\|_{H^m}\to0$ uniformly in $t\in[0,1]$ due to $\|v^r_0-v_0\|_{H^m}\to0$ and the continuity of the Riemannian exponential map.

If the metric is three times differentiable, the Riemannian exponential is even differentiable so that $\|v^r_t-v_t\|_{H^m}\leq C(\NHm{v_0})\|v^r_0-v_0\|_{H^m}\leq\tilde C(\|v_0\|_{H^{m+1}})r$
for constants $C(\NHm{v_0}),\tilde C(\|v_0\|_{H^{m+1}})$ depending only on the norm of $v_0$.
Furthermore, the right-hand side of \eqref{eqn:rR_estimate} tends to zero faster than $R^{\frac32-k}$ (with a constant depending on $\NHm{v_0}$),
which implies the second claim.
\end{proof}

\begin{figure}
\centering
\begin{tikzpicture}
  \begin{loglogaxis}
    [width=.8\linewidth,height=.5\linewidth,
    xmin=4, xmax=64, domain=4:64,
    xlabel={frequency cutoff $R$},ylabel={error $\|e_1\|_{H^2}$},
    xticklabels={},
    extra x ticks={4,8,16,32,64},extra x tick labels={4,8,16,32,64},
    legend pos=south west,legend style={font=\scriptsize}]
    \addplot[gray,domain=8:32,forget plot] {1};
    \addplot[gray,domain=8:32,forget plot] {1/x};
    \addplot[gray,domain=8:32,forget plot] {1/x^2};
    \addplot[gray,domain=8:32,forget plot] {1/x^3};
    \addplot[mark=x,mark options={solid},black,dotted,thick] coordinates {(64,4.6161)(32,5.3727)(16,5.5677)(8,5.3352)(4,5.2693)};
    \addplot[mark=x,mark options={solid},black,dashdotted,thick] coordinates {(64,1.5275)(32,1.5290)(16,1.8744)(8,2.4835)(4,2.4092)};
    \addplot[mark=x,mark options={solid},black,dashed,thick] coordinates {(64,1.7057e-3)(32,8.0265e-3)(16,4.1156e-2)(8,1.3020e-1)(4,2.5303e-1)};
    \addplot[mark=x,mark options={solid},black,thick] coordinates {(64,6.2473e-6)(32,5.4148e-5)(16,4.7420e-4)(8,3.9541e-3)(4,1.9867e-2)};
    \legend{$v_0\in H^3(\T^d;\R^d)$,$v_0\in H^4(\T^d;\R^d)$,$v_0\in H^5(\T^d;\R^d)$,$v_0\in H^6(\T^d;\R^d)$}
  \end{loglogaxis}
\end{tikzpicture}
\caption{Numerical validation of the convergence result \cref{thm:convergence} in $d=2$ space dimensions for $m=3$.
The numerically estimated error of the discretized solution is shown as a function of the truncation cutoff $R$ for experiments with different Sobolev regularity of the initial condition $v_0$.
The grey lines indicate the rates $R^0$, $R^{-1}$, $R^{-2}$, and $R^{-3}$.}
\label{fig:convRates}
\end{figure}
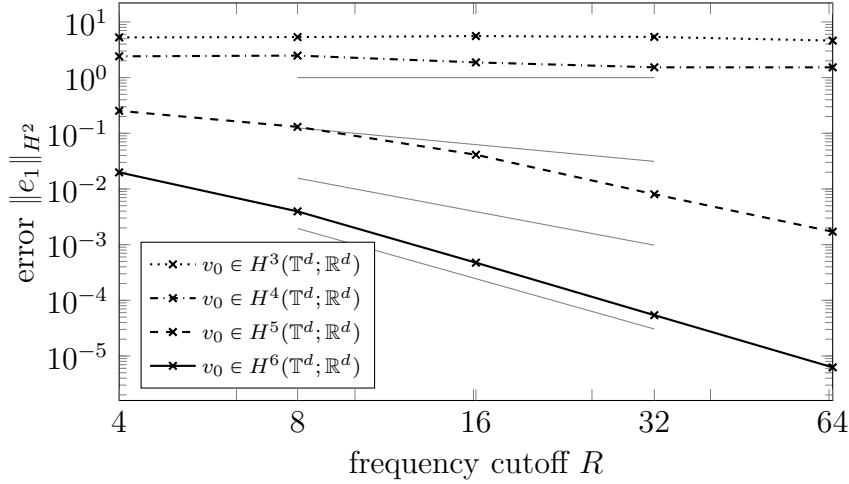

\begin{remark}[Convergence on $\R^d$]
The arguments hold verbatim also when replacing the domain $\T^d$ by $\R^d$ and interpreting $\trunc$ as the truncation operator on continuous Fourier space.
However, this is of course of little numerical relevance since the resulting truncated ODE \eqref{eqn:geodesicEquationLDDMMFourier} is still infinite-dimensional.

In certain situations the geodesic ODE on $\R^d$ can actually be reduced to the geodesic ODE on $\T^d$.
Indeed, an integral of \eqref{eqn:EPDiff} with momentum $\rho$, velocity $v$ and flow $\phi$ is given by $\rho_t=\adjointMap_{\phi_t}^*\rho_0$,
from which one can quickly see that if $\rho_0$ has compact support, then so does $\rho_t$ for all $t\in[0,1]$ \cite[Sec.\,4.3.1]{MiTrYo06}.
By rescaling and shifting coordinates we may assume $\rho_t$ to stay supported within $[0,1)^d$.
Hence, by appropriately choosing the operator $\Riesz_{\T^d}$ (the Riesz isomorphism to be used on the torus; the subscript distinguishes it from the original operator $\Riesz$ on $\R^d$)
one obtains the identical evolution of $\rho_t$ on $\T^d$ as on $\R^d$:
Operator $\Riesz_{\T^d}$ would be the composition of an extension (of $\rho_t$) from $[0,1)^d$ to $\R^d$ by zero
(in Fourier space this would mean to interpret the discrete Fourier sum as a sum of Dirac measures and convolving this measure with the Fourier transform $\hat\chi$ of the characteristic function $\chi$ of $[0,1)^d$),
an application of $\Riesz$ (which then yields $v_t$ on $\R^d$),
a restriction (of $v_t$) to $[0,1)^d$ (in Fourier space another convolution with $\hat\chi$),
and a reinterpretation as periodic function (by the Poisson summation formula, the discrete Fourier coefficients would just be the sampling of the continuous ones).
Unfortunately, our error analysis only applies in this setting if $\Riesz_{\T^d}$ is a Fourier multiplier, which for natural choices (such as $\Riesz=\L^{-1}$ with $\L=(1-\Delta)^m$) is not the case.

For initial momentum of unbounded support the situation of course becomes even more involved.
For a numerical approximation one would need to periodize momentum or velocity (typically via the Poisson summation formula)
in addition to truncating in Fourier space and then analyze the error inflicted by both.
A possible direction would be to truncate $\rho_0$ in space and investigate the resulting error,
after which in a second step the above reduction from $\R^d$ to $\T^d$ is applied.
\end{remark}

For a numerical validation of \cref{thm:convergence} we solve the (space-discrete) geodesic ODE \eqref{eqn:LDDMMapproximateLie} on $[0,1]$ for $R\in\{16,32,64,128\}$
using an explicit $6$-stage Runge--Kutta method (the $5$th order consistent $6$-stage part of the Dormand--Prince method) with step size $2^{-15}$
(by using different step sizes we checked that the error from the time discretization is negligibly small compared to the error from the spectral space discretization).
The result for $R=128$ is taken as a substitute for the true solution, and errors for smaller $R$ are computed with respect to that solution.
The experiments are performed in $d=2$ space dimensions with $m=3$ and the inner product $\IPHm{\cdot}{\cdot}$ defined via the corresponding differential operator $\L=(1-\Delta)^m$.
To obtain initial conditions $v_0$ of different Sobolev regularity we first define a complex-valued $H^0$-function $w$
by choosing each Fourier coefficient $\hat w(\xi)$ randomly and independently from $[0,1/\sqrt{1+|\xi|^2}/\log(2+|\xi|^2)^{1/2+\epsilon}]$ for $\epsilon=0.1$.
Via $\hat u(\xi)=\hat w(\xi)+\overline{\hat w(-\xi)}$ we obtain a real-valued $H^0$-function $u$, and we simply set
\begin{equation*}
\hat v_0(\xi)=\hat u(\xi)/(1+|\xi|^2)^{s/2}
\end{equation*}
to obtain a $H^s$-regular initial condition $v_0$.
\Cref{fig:convRates} shows that the numerically observed convergence rates for different initial Sobolev regularity are better than the one of \cref{thm:convergence} by one order;
it may thus be that, generically (or at least for the above choice of initial conditions $v_0$), convergence rates are better, something we leave for future investigation.

Let us finally comment on why we restricted to operators $\L$ that are Fourier multipliers:
First off, otherwise $P_t=\L V_t$ for a bandlimited $V_t$ will not be bandlimited so that one needs to approximate $\L$ (by $\trunc\L$ or in a different way) to obtain a working scheme.
However, central to our analysis was the estimate from \cref{thm:productsInB}, and it is unclear how to approximate $\L$ with a bandlimit while preserving such an estimate.

\paragraph{Acknowledgements.}
This work was supported by the Deutsche Forschungsgemeinschaft (DFG, German Research Foundation) via project 431460824 -- Collaborative Research Center 1450
and via Germany's Excellence Strategy project 390685587 -- Mathematics M\"unster: Dynamics-Geometry-Structure.

\bibliographystyle{alpha}
\bibliography{references}

\end{document}